\documentclass[twocolumn,fleqn]{svjour3}         
\smartqed  
\usepackage{amsmath,amsfonts,amssymb,afterpage,color}
\usepackage{graphicx}
\usepackage{epsfig}
\usepackage{epsf}
\graphicspath{
{./figs/}
{../figs/}
}

\def\text#1{{\rm #1}}

\def\calg#1{{\mathcal #1}}

\def\eig{eigenvalue}
\def\eigs{eigenvalues}

\def\oalp{\mathcal{O}(m^{-1})}

\def\Re{\mathbb R}

\def\hm{\overline{h}_m}

\def\Lamalp{\tilde\Lambda(m)}
\def\delt{\tilde{\delta}}
\def\deltt{\tilde{\delta}^{T}}

\def\oalp{\cO(m^{-1})}

\newcommand{\cT}{\mathcal{T}}

\newcommand{\cP}{\mathcal{P}}

\newcommand{\cO}{\mathcal{O}}
\newcommand{\cN}{\mathcal{N}}

\begin{document}

\title{Robust multigrid preconditioners for cell-centered
finite volume discretization of the high-contrast diffusion equation 
}

\titlerunning{Robust multigrid preconditioner for high-contrast diffusion equation}        

\author{Burak Aksoylu \and Zuhal Yeter}

\authorrunning{Burak Aksoylu \and Zuhal Yeter} 

\institute{B. Aksoylu (corresponding author) \and Z. Yeter \at
Department of Mathematics \&  \\
Center for Computation and Technology, \\
Louisiana State University\\
216 Johnston Hall, Baton Rouge LA, 70803 USA\\
\email{burak@cct.lsu.edu,~zyeter1@cct.lsu.edu}
}
\date{
{\sc Submitted: April 11, 2009}}

\maketitle

\pagestyle{headings}
\markboth
{B. Aksoylu, Z. Yeter}
{Robust cell-centered multigrid preconditioners for high-contrast 
diffusion equation}

\begin{abstract}

  We study a conservative 5-point cell-centered finite volume
  discretization of the high-contrast diffusion equation. We aim to
  construct preconditioners that are robust with respect to the
  magnitude of the coefficient contrast and the mesh size
  simultaneously.  For that, we prove and numerically demonstrate the
  robustness of the preconditioner proposed by Aksoylu et al.  (2008,
  Comput. Vis. Sci. 11, pp. 319--331) by extending the devised
  singular perturbation analysis from linear finite element
  discretization to the above discretization.  The singular
  perturbation analysis is more involved than that of finite element
  because all the subblocks in the discretization matrix depend on the
  diffusion coefficient. However, that dependence is eliminated
  asymptotically. This allows the same preconditioner to be utilized
  due to similar limiting behaviours of the submatrices; leading to a
  narrowing family of preconditioners that can be used for different
  discretizations---a desirable preconditioner design goal.  We
  compare our numerical results to standard cell-centered multigrid
  and observe that performance of our preconditioner is independent of
  the utilized prolongation operators and smoothers.

  As a side result, we also prove that the solution over the
  highly-diffusive island becomes constant asymptotically.
  Integration of this qualitative understanding of the underlying PDE
  to our preconditioner is the main reason behind its superior
  performance.  Diagonal scaling is probably the most basic
  preconditioner for high-contrast coefficients. Extending the matrix
  entry based spectral analysis introduced by Graham and Hagger, we
  rigorously show that the number of small eigenvalues of the
  diagonally scaled matrix depends on the number of isolated islands
  comprising the highly-diffusive region.

\end{abstract}

\section{Introduction}

We advocate that qualitative understanding of the PDE operators and
their dependence on the coefficients is essential for designing
preconditioners. Since, the performance, hence the robustness, of a
preconditioner depends essentially on the degree to which the
preconditioner approximates the properties of the underlying
PDE. Therefore, designing preconditioners involves a process that
draws heavily upon effective utilization of tools from operator theory
as well as singular perturbation analysis (SPA).  In the operator
theory framework, Aksoylu and Beyer~\cite{AkBe2008,AkBe2009} have
studied the diffusion equation with rough coefficients. The roughness
of coefficients creates serious complications.  For instance, it was
shown in~\cite{AkBe2008} that the standard elliptic regularity in the
smooth coefficient case fails to hold. Moreover, the domain of the
diffusion operator heavily depends on the regularity of the
coefficients.

Roughness of PDE coefficients causes loss of robustness of
preconditioners.  We aim to establish robustness with respect to the
magnitude of the coefficient contrast and the mesh size
simultaneously.  In that regard, SPA provides valuable insight into
the qualitative nature of the underlying PDE. In the case of linear
finite element (FE), Aksoylu et al.~\cite{AGKS:2007} devised a SPA on the
matrix entries to study the robustness of the same preconditioner
$B_{AGKS}$ under consideration in this article.  SPA turned out to be
an effective tool in analyzing certain behaviors of the discretization
matrix $K$ such as the asymptotic rank, decoupling, low-rank
perturbations of the resulting submatrices. This information in turn is
exploited to accomplish dramatic computational savings. 
In~\cite{AGKS:2007}, we also provided a rigorous convergence analysis
of $B_{AGKS}$.  Hence, with the insights provided by SPA and the
operator theory, we are in control of the effectiveness and
computational feasibility at the same time.

The preconditioner $B_{AGKS}$ originates from the family of robust
preconditioners constructed by Aksoylu and Klie~\cite{AkKl:2007} for
the cell-centered finite volume (FV) discretization of the
high-contrast diffusion equation for porous media flow related
applications based on the two-point approximation scheme studied
in~\cite{aarnesGimseLie2007}.  However, for the original variants of
$B_{AGKS}$, robustness with respect to the contrast size was the main
design feature.  Hence, the emphasis in~\cite{AkKl:2007} was mostly on
deflation methods (that are used as stage-two preconditioners) based
on the Krylov subspace solvers. As in~\cite{AGKS:2007}, in this
article, we incorporate (cell-centered) multigrid preconditioners to
restore robustness with respect to the mesh size.  Hence, the main
purpose of this article is to extend the preconditioner $B_{AGKS}$ and
the related analysis to a 5-point conservative cell-centered FV
discretization.  From the flow application perspective, maintaining
the continuity of the flux across the control volume interfaces
ensures the highly popular and crucial discretization property of
\emph{local mass conservation}. Our interest in flow applications
is the main reason behind conducting research in the direction of
mass conservative discretizations.

We prove and numerically demonstrate that the very same preconditioner
$B_{AGKS}$ that is used for FE discretization can also be used with
minimal modification for FV discretizations.  This was possible by the
help of SPA because we have identified similar algebraic features of
the discretization matrices between linear FE and FV methods.  The
same preconditioner can be utilized due to similar limiting behaviours
of the submatrices.  This observation leads to a narrowing family of
preconditioners that can be used for different discretizations.
Therefore, we have accomplished to construct a preconditioner that is
compatible with and equally effective under different discretizations,
and this is a desirable feature in the design and construction of
preconditioners.  In addition, extension to FV discretization does not
spoil $B_{AGKS}$'s inherit algebraic nature.  The first algebraic
phase involves partitioning of the degrees of freedom (DOF) into a set
corresponding to a high-diffusivity and a low-diffusivity region. For
high enough contrast, we can still obtain the partitioning by
examining the diagonal entries of $K$.  This simple algebraic
examination rules out the standard multigrid requirement of the
coarsest mesh resolving the boundary of the island.

The diffusion equation with discontinuous coefficients is known as the
\emph{interface problem} in the computational fluid dynamics
community.  There has been intense research activity on the interface
problem. It is such a well-established problem that one can find it in
the text books dedicated to multigrid methods, for instance, by
Wesseling~\cite{Wesseling.P2001} and Trottenberg, Oosterlee, and
Sch\"uller \cite{Trottenberg.U;Oosterlee.C;Schuller.A2001}.  Mohr and
Wienands~\cite{mohrWienands2004} revisited the cell-centered multigrid
(CCMG) preconditioner and attributed the pioneering CCMG to
Wesseling~\cite{wesseling1988}; see~\cite{mohrWienands2004} and the
references therein for further review of CCMG.

For interface problems, there have been many attempts to construct problem
independent prolongation and restriction operators to accommodate the
the roughness of coefficients. Among the variational approaches,
Wesseling~\cite{wesseling1988} and Wesseling and
Khalil~\cite{wesselingKhalil1992} constructed such prolongations for
high-contrast coefficients, whereas Kwak \cite{kwak1999} studied
medium-contrasts.  Kwak and Lee~\cite{kwakLee2004} proposed problem
dependent prolongations for medium-contrasts.  Among the
non-variational approaches, Ewing and Shen~\cite{ewingShen1993}
examined medium-contrast with piecewise constant prolongation and
restriction operators. In our CCMG and $B_{AGKS}$ implementations, we
prefer problem independent prolongation operators utilizing Galerkin
variational approach based on Wesseling and
Khalil~\cite{wesselingKhalil1992} and bi-linear interpolation;
see~\cite[Figure 2]{mohrWienands2004}
or~\cite[p. 72]{Wesseling.P2001}.

Diagonal scaling is probably the most basic preconditioner for
discretizations with high-contrast coefficients. Although diagonal
scaling has no effect on the asymptotic behaviour of the condition
number, Graham and Hagger~\cite{GrHa:99} observed, in the case of FE,
that spectrum of the diagonally scaled matrix $A$ enjoys better
clustering than that of $K$.  The spectrum of $A$ is bounded from
above and below except a single eigenvalue in the case of a single
isolated highly-diffusive island. On the other hand, the spectrum of $K$
contains eigenvalues approaching infinity with cardinality depending
on the number of DOF contained within highly-diffusive island. For
faster convergence, the Krylov subspace solver favors the clustering
provided by diagonal scaling.  Based the matrix entry based
spectral analysis introduced by Graham and Hagger~\cite{GrHa:99} for linear FE,
we extend the spectral analysis to cell-centered FV discretization and
obtain a similar spectral result for $A$. Namely, the number of small
eigenvalues of $A$ depends on the number of isolated islands
comprising the highly-diffusive region. This rigorous analysis is
presented in Section~\ref{sec:diagScaling}.

We extend the devised SPA from FE to FV discretization in order
to explain the properties of the submatrices related to $K(m)$.
In particular, SPA of $K_{HH}(m)$, as diffusivity $m \to \infty$, has
important implications for the behaviour of the Schur complement $S(m)$
of $K_{HH}(m)$ in $K(m)$. Namely,
\begin{equation*}
S(m) = S^\infty + \mathcal{O}(m^{-1}) \ ,
\end{equation*}
where $S^\infty$ is a low rank perturbation of $K_{LL}^\infty$, i.e.,
the limiting $K_{LL}(m)$. The rank of the perturbation depends on the
number of disconnected components comprising the highly-diffusive
region.  This special limiting form of $S(m)$ allows us to build a
robust approximation of $S(m)^{-1}$ by merely using solvers for
$K_{LL}^\infty$ by the help of the Sherman-Morrison-Woodbury
formula. In Section~\ref{mainResultsSPA}, by using SPA, the asymptotic
behaviour of submatrices is provided and the final convergence proof
of $B_{AGKS}$ is based on that result.

The remainder of the article is structured as follows.  In Section
\ref{sec:qualitativeNature}, we reveal the qualitative nature of the
solution of the high-contrast diffusion equation and the resulting
decoupling in the solution. This valuable insight is provided by
SPA. In Section \ref{sec:subdomainDefl}, we highlight the
implementation aspects of $B_{AGKS}$. In addition, we present how
subdomain deflation gives rise to a dramatic computational saving due
to reduction to a block diagonal system.  Finally, in Section
\ref{sec:numericalExperiments}, we numerically demonstrate the
simultaneous robustness of $B_{AGKS}$ with respect to magnitude of the
coefficient contrast and the mesh size. We also provide comparisons to
CCMG preconditioner with different prolongation operators and
smoothers.

\section{The underlying PDE and the linear system}

\begin{figure}[htbp]
\centering{
\includegraphics[width=1.25in]{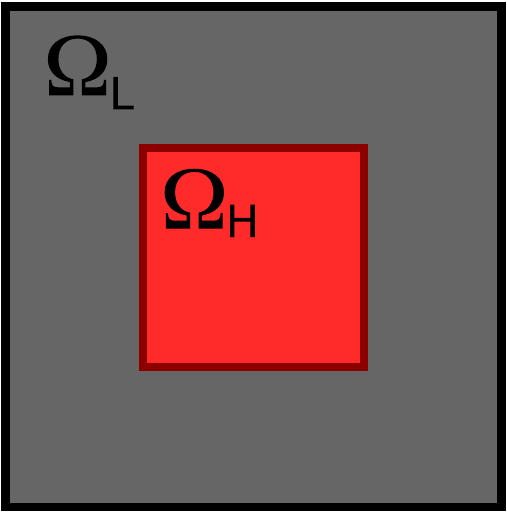}}
\caption{$\Omega = \overline{\Omega}_H \cup \Omega_L$ where $\Omega_H$
  and $\Omega_L$ are high and low diffusivity regions,
  respectively.\label{fig:domain1}}
\end{figure}

We study the following high-contrast diffusion equation.
\begin{equation} \label{mainProblem}
\left\{
\begin{array}{rcll}
 -\nabla\cdot(\alpha\nabla u) & = & f~~& \text{in}~ \Omega \\
u & = & 0~~ & \text{on}~ \partial \Omega  
\end{array}
\right.
\end{equation}
where $\Omega \subset \Re^d,~d=2,3$.  The coefficient $\alpha(x)$ may
vary over many orders of magnitude in an unstructured way on $\Omega$.
Problems with high-contrast coefficients are ubiquitous in porous
media flow applications. Many examples of this kind arise in
groundwater flow and oil reservoir simulation; see for example the
comprehensive overviews \cite{A:07,GD:05,MY:06,Noetinger05}.
Consequently, development of efficient solvers for high-contrast
heterogeneous media has been an active area of research, specifically
in the setting of multiscale solvers
~\cite{AarnesHou:02,GrHa:99,GrLeSc:2007,GrSc:2007,ScVa2007}.  In
addition, the fictitious domain method and composite materials are
also sources of rough coefficients; see the references
in~\cite{KnWi:2003}.  Important current applications deal with
composite materials whose components have nearly constant diffusivity,
but vary by several orders of magnitude.  In composite material
applications, it is quite common to idealize the diffusivity by a
piecewise constant function.  Likewise, we restrict the diffusion
process to a \emph{binary regime} (see Figure \ref{fig:domain1}) in
which the coefficient $\alpha$ is a piecewise constant function with
the following values:
\begin{equation*}
\alpha(x) =
\begin{cases}
m \gg 1, & x \in \Omega_H, \\
1, & x \in \Omega_L.
\end{cases}
\end{equation*}
Due to the atomistic structure of matter, the physical treatment of
diffusion involves regular ($C^\infty$-) diffusivity.  Aksoylu and
Beyer~\cite{AkBe2008} showed that the idealization of diffusivity by
piecewise constant coefficients is meaningful by showing a continuous
dependence of the solutions on the diffusivity.

Let us denote the linear system arising from the finite
volume discretization by:
\begin{equation} \label{mainLinearSys}
K(m)~x(m) = b.
\end{equation}
Let $\Omega$ be decomposed with respect to diffusivity value as
\begin{equation} \label{subregionDecomp}
\Omega = \overline{\Omega}_H \cup \Omega_L,
\end{equation}
where $\Omega_H$ and $\Omega_L$ denote the high and low diffusivity
regions, respectively.  When $m$-dependence is explicitly stated and
the discretization system \eqref{mainLinearSys} is decomposed with
respect to \eqref{subregionDecomp}, i.e., the magnitude of the
coefficient values, we arrive at the following $2 \times 2$ block
system:
\begin{equation} \label{2x2blockSys}
\left[
\begin{array}{ll}
K_{HH}(m) & K_{HL}(m) \\ K_{LH}(m) & K_{LL}(m)
\end{array}
\right]
\left[ \begin{array}{c} x_H(m) \\ x_L(m) \end{array} \right]
= \left[ \begin{array}{c} b_H \\ b_L \end{array} \right].
\end{equation}
Note that all the subblocks in \eqref{2x2blockSys} have
$m$-dependence. This is the main difference from the FE
discretization in which only the $HH$-subblock has
$m$-dependence. Hence, the perturbation analysis of the FV
discretization becomes more involved than that of the FE
discretization.

The exact inverse of $K$ can  be written as:
\begin{eqnarray}
K^{-1}(m) & = &
{ \left[ \begin{array}{cc}  I_{HH} & ~~- K_{HH}^{-1}(m) K_{HL}(m) \\
0 & I_{LL}\end{array}\right]} \nonumber \\
&& \times
{ \left[ \begin{array}{cc}    K_{HH}^{-1}(m)& 0 \\
0& ~~S^{-1}(m) \end{array}\right]} \nonumber \\
&& \times
{ \left[ \begin{array}{cc}  I_{HH} & 0 \\    -K_{LH}(m) K_{HH}^{-1}(m)& ~~I_{LL}
\end{array}\right]}, \label{eq:exact}
\end{eqnarray}
where $I_{HH}$ and $I_{LL}$ denote the
identity matrices of the appropriate dimension and
the $S(m)$ is the Schur complement of $K_{HH}(m)$ in $K(m)$ given by 
\begin{equation} \label{SchurComplement1}
S(m) = K_{LL}(m) - K_{LH}(m) K_{HH}^{-1}(m) K_{HL}(m).
\end{equation} 

Let $\cN_{HH}$ denote the Neumann matrix corresponding to the pure
Neumann problem for the Laplace operator on $\Omega_H$.  We write an
important decomposition which will be used in the analysis to come:
\begin{equation}
\label{eq:KHH}
K_{HH}(m) = m \, \cN_{HH} + \Delta(m) \ .
\end{equation}

\section{The cell-centered finite volume discretization}

We start by giving the cell-centered FV discretization for
the $5$-point stencil used in this article; for more details
see~\cite{eymardGallouetHerbin2000_fvmBook}.
Let $\cT = V_{i,j}$ with $ i,j = 1,\ldots, n^{1/2}$ be the mesh and 
the control volume be defined by
\begin{equation*}
  V_{i,j} = [x_{i-1/2}, x_{i+1/2}] \times [y_{j-1/2}, y_{j+1/2}].
\end{equation*}
The FV scheme is constructed by integrating 
\eqref{mainProblem} over each control volume $V_{i,j}$, which yields
\begin{eqnarray*}
&-& \int_{y_{j-1/2}}^{y_{j+1/2}} \alpha_{i+1/2,j} ~ u_x(x_{i+1/2},y) \, dy \\
&+& \int_{y_{j-1/2}}^{y_{j+1/2}} \alpha_{i-1/2,j} ~ u_x(x_{i-1/2},y) \, dy \\
&+& \int_{x_{i-1/2}}^{x_{i+1/2}} \alpha_{i,j-1/2} ~ u_y(x,y_{j-1/2}) \, dx \\
&-& \int_{x_{i-1/2}}^{x_{i+1/2}} \alpha_{i,j+1/2} ~ u_y(x,y_{j+1/2}) \, dx \\
&=& \int_{V_{i,j}} f(x,y) \, dxdy.
\end{eqnarray*} 

Thus, the discretization scheme is
\begin{equation*}
F_{i+1/2,j} - F_{i-1/2,j} + F_{i,j+1/2}-F_{i,j-1/2} =h_{i,j}f_{i,j},
\end{equation*}
where $h_{i,j} = k_x\times k_y$, and $f_{i,j}$ is the mean value of $f$ 
over $V_{i,j}$, and where 
\begin{eqnarray*}
F_{i+1/2,j} &=& 
-\frac{k_y}{k_x} \ \overline{h}_{\alpha_i}\left\{u(x_{i+1},y_j)-u(x_{i},y_j)\right\}, \\
F_{i,j+1/2} &=& 
-\frac{k_x}{k_y} \ \overline{h}_{\alpha_j}\left\{u(x_{i},y_{j+1})-u(x_{i},y_{j})\right\}, 
\end{eqnarray*}
and
\begin{equation} \label{harmonicMeans}
\overline{h}_{\alpha_i} = \frac{2 \, \alpha_{i,j} \, \alpha_{i+1,j}}{\alpha_{i,j}+\alpha_{i+1,j}},~~
\overline{h}_{\alpha_j} = \frac{2 \, \alpha_{i,j} \, \alpha_{i,j+1}}{\alpha_{i,j}+\alpha_{i,j+1}}.
\end{equation}
Note that $\overline{h}_{\alpha_i}$ and $\overline{h}_{\alpha_j}$ are the harmonic means of the diffusion
coefficients for the adjacent control volumes in $x$ and $y$
directions respectively. The discretization formula can be written
explicitly as follows:
\begin{eqnarray} 
&-& \overline{h}_{\alpha_i}u(x_{i+1},y_j) - \overline{h}_{\alpha_{i-1}}u(x_{i-1},y_j) \nonumber \\
&+& (\overline{h}_{\alpha_i} + \overline{h}_{\alpha_{i-1}} + \overline{h}_{\alpha_j} + \overline{h}_{\alpha_{j-1}})u(x_{i},y_j) \label{FVMdiscretization} \\
&-& \overline{h}_{\alpha_j}u(x_{i},y_{j+1}) -\overline{h}_{\alpha_{j-1}}u(x_{i},y_{j-1}) = h_{i,j}f_{i,j} \nonumber.
\end{eqnarray}

In our binary diffusivity regime, for notational convenience, we
denote the harmonic mean by
\begin{equation} \label{hMean}
\overline{h}_m := \frac{2\,m}{m+1}.
\end{equation}
The \emph{harmonic mean} is used to ensure the \emph{continuity of the
  flux} across the control volume interfaces. As a result, this flavor
of finite volume discretization enjoys the desirable property of mass
conservation.

One can write the discretization matrix entries \emph{a priori} due
to the formula \eqref{FVMdiscretization}. Hence, in 2D, we explicitly state
each contribution of the off-diagonals to the diagonal entry values
in the following:
\begin{eqnarray}
&& [K(m)]_{pp} = \label{matrixEntryValues} \\
&&
\left\{
\begin{array}{llllllll}
m & + & m & + &   m & + &   m, &~p \in I_{\Omega_{1}},\\ 
m & + & m & + &   m & + & \hm, &~p \in \Gamma_{\Omega_{1}}~\textrm{and~non-corner}, \\
m & + & m & + & \hm & + & \hm, &~p \in \Gamma_{\Omega_{1}}~\textrm{and~corner},\\
1 & + & 1 & + &   1 & + & \hm, &~p \in \Gamma_{\Omega_{2}}, \\
1 & + & 1 & + &   1 & + &   1, &~p \in I_{\Omega_{2}}~\textrm{and~strictly~interior}, \\
5, &&&&&& &~p \in I_{\Omega_{2}}~\textrm{and~non-corner~bdry}, \\
6, &&&&&& &~p \in I_{\Omega_{2}}~\textrm{and~corner~bdry}.
\end{array} \right. \nonumber
\end{eqnarray}

\subsection{Comparison of finite element and finite volume discretizations
on a 1D example}

\begin{figure}[h]
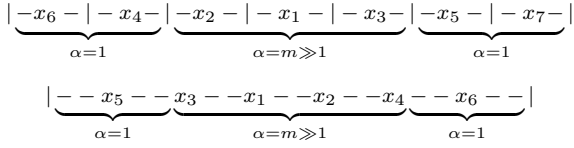

$$
|\underbrace{-x_6-|-x_4-}_{\alpha = 1}
|\underbrace{-x_2-|-x_1-|-x_3-}_{\alpha = m \gg 1}
|\underbrace{-x_5-|-x_7-}_{\alpha = 1}|
$$
$$
|\underbrace{--x_5--}_{\alpha =1}
\underbrace{x_3--x_1--x_2--x_4}_{\alpha = m\gg 1}
\underbrace{--x_6--}_{\alpha = 1}|
$$
\caption{(Top) The finite volume mesh where the cell-centers are
denoted by $x_i,~i=1, \ldots, 7.$ (Bottom) 
The finite element mesh where the vertex locations are
denoted by $x_i,~i=1, \ldots, 6.$  \label{1Dmeshes}}
\end{figure}

We explicitly provide the discretization matrix utilizing the
FV method in \eqref{FVMdiscretization} corresponding to
\eqref{mainProblem}.  The domain is chosen to be $\Omega:=(0,1)$ with
the highly-diffusive island $\Omega_H := (2/7, 5/7)$.  The mesh consists
of 7 cells. The cells and cell-centers are denoted by $V_1, \ldots, V_7$ and
$x_1, \ldots, x_7$, respectively. See Figure \ref{1Dmeshes}.

The corresponding submatrices in \eqref{2x2blockSys} are given below:
\begin{eqnarray*}
K_{HH}(m) & = & \left[
 \begin{matrix}
 2m & -m & -m \\ -m & m + \overline{h}_m & 0 \\ -m & 0 &
 m+\overline{h}_m \\
 \end{matrix} \right],\\
K_{HL}(m) & = & \left[
\begin{matrix}
0 & 0 & 0 & 0 \\ 
-\overline{h}_m & 0 & 0 & 0 \\ 
0 & -\overline{h}_m & 0 & 0 \\
\end{matrix}  \right] = K_{LH}^{ T}(m), \\
K_{LL}(m) & = & \left[
\begin{matrix}
   1+\overline{h}_m & 0 & -1 & 0 \\ 0 & 1+\overline{h}_m & 0 & -1
   \\ -1 & 0 & 3 & 0 \\ 0 & -1 & 0 & 3 \\
 \end{matrix}  \right].
\end{eqnarray*}

Moreover, from \eqref{eq:KHH}, we obtain
\begin{eqnarray*}
\mathcal{N}_{HH}(m) & = & \left[
 \begin{matrix}
 2m & -m & -m \\ -m & m & 0 \\ -m & 0 & m \\
 \end{matrix}
 \right], \\
 \Delta(m) & = & \left[
 \begin{matrix}
 0 & 0 & 0 \\ 0& \overline{h}_m & 0 \\ 0 & 0 &\overline{h}_m \\
 \end{matrix}
 \right].
\end{eqnarray*}
We readily see that the $m$-dependence of the matrices 
$K_{LH}(m), K_{HL}(m)$,  $K_{LL}(m)$, and $\Delta(m) $ is eliminated
asymptotically. For instance, 
\begin{eqnarray*}
\Delta(m) & = & \Delta^\infty + \oalp,\\
& = & 
\left[ 
\begin{matrix}
0 & 0 & 0 \\ 
0 & 2 & 0 \\ 
0 & 0 & 2
\end{matrix} 
\right] + \oalp.
\end{eqnarray*}

By maintaining a similar configuration (see Figure \ref{1Dmeshes})
used for the FV case, the corresponding submatrices in
\eqref{2x2blockSys} for the linear FE discretization are given
below. DOF that lie on the interface are always included to the DOF of
the highly-diffusive island.  Hence, only $K_{HH}^{FE}$ subblock has
$m$-dependence.

\begin{eqnarray*}
K_{HH}^{FE}(m) & = & \left[
\begin{matrix}
  2m & -m & -m & 0\\ -m & 2m & 0 &-m \\ -m & 0 &
  m+1 & 0 \\ 0 &-m & 0 & m+1 \\
\end{matrix} \right],\\
K_{HL}^{FE} & = & \left[
\begin{matrix}
  0 & 0 \\ 0 & 0 \\ -1 & 0 \\ 0 &-1 \\
\end{matrix} \right] = K_{LH}^{FE^T},\\
K_{LL}^{FE} & = & \left[
\begin{matrix}
  2 & 0 \\ 0 & 2 \\
\end{matrix} \right].
\end{eqnarray*}

\begin{eqnarray*}
m\mathcal{N}_{HH}^{FE} & = & \left[
\begin{matrix}
  2m & -m & -m & 0\\ -m & 2m & 0 &-m \\ -m & 0 &
  m & 0 \\ 0 &-m & 0 & m \\
\end{matrix}
\right],\\
\Delta^{FE} & = & \left[
\begin{matrix}
0 & 0 & 0 & 0 \\ 
0 & 0 & 0 & 0 \\ 
0 & 0 & 1 & 0 \\ 
0 & 0 & 0 & 1
\end{matrix}
\right].
\end{eqnarray*}

\section{Spectral analysis of the diagonally scaled finite volume 
discretization matrix} \label{sec:diagScaling}

Let $\{1, \ldots, s\}$ denote the index of islands in the domain.  Let
$N_{k},~ k = 1, \ldots, s$ be the FV discretization matrix of
$-\Delta$ on the k-th island $\Omega_{k}$ with respect to $\cT_{i}$
with homogeneous pure Neumann boundary condition.  Let $\mathcal{C}$ denote
the set of all DOF in $\Omega$.  We will analyze the behavior of the
spectrum of the symmetrically scaled discretization matrix
\begin{equation}
\label{diagscaled}
A(m) := 
\left( \textrm{diag} K(m) \right)^{-1/2} K(m) 
\left( \textrm{diag} K(m) \right)^{-1/2}.
\end{equation}
We start by examining the dependence of the $(p, q)$-th entry of
$K(m)$ on $m$, for each $p \in \mathcal{C}$.
Let $\Omega_{s+1}$ denote the outer region of islands, i.e., 
\begin{equation*}
\Omega_{s+1} = \Omega \setminus \bigcup_{k=1, \ldots, s} \overline{\Omega}_k.
\end{equation*}
We denote the cell-centers that are adjacent to
$\Omega_{s+1}$ and the ones that are interior to
$\Omega_{k},~k=1, \ldots, s$ by $\Gamma_{\Omega_{k}}$ and
$I_{\Omega_{k}}$, respectively. On the other hand, the cell-centers in
the outer region $\Omega_{s+1}$ that are adjacent to the islands
$\Omega_{k}, k=1, \ldots, s$ are denoted by $\Gamma_{\Omega_{s+1}}$.

We define the following index set for $p, q \in \mathcal{C}$ with $p \ne q$: 
\begin{eqnarray*}
\mathcal{I}(p \wedge q) & := &
\begin{cases}
k,   & p~\text{and}~q \in \Omega_{k},~k=1, \ldots, s \\ 
s+1, & p~\text{or}~ q \in \Omega_{s+1},
\end{cases}\\
\mathcal{I}(p) & := &
\begin{cases}
k,      & p \in I_{\Omega_{k}},~k=1, \ldots, s \\ 
s+1, k, & p \in \Gamma_{\Omega_{k}},~k=1, \ldots, s \\ 
s+1,    & p \in \Omega_{s+1}.
\end{cases}
\end{eqnarray*}
Also, we define $\mathcal{I}(p \wedge p) = \mathcal{I}(p)$.

Note that the discretization matrix and its entries can be written as follows:
\begin{equation*}
K(m) = \sum_{k=1}^s m N_k(1) + N_{s+1}(m),
\end{equation*}
\begin{equation*}
[K(m)]_{pq} = \sum_{\ell \in \mathcal{I}(p \wedge q)}
\alpha_{\ell}(m) [N_{\ell}]_{pq},
\end{equation*} 
where  
\begin{equation}
\alpha_{\ell}(m) = 
\begin{cases}
m, & \ell=1, \ldots, s \\
1, & \ell = s+1,
\end{cases}
\end{equation}
and by abuse of notation,
we have defined $N_{\ell} := N_{\ell}(1)$ for 
$\ell=1, \ldots, s$ and $N_{s+1} := N_{s+1}(m)$. 
Then, 
\begin{eqnarray*}
& & [A(m)]_{pq} \\
& & = [K(m)]_{pp}^{-1/2}~[K(m)]_{pq}~
[K(m)]^{-1/2}_{qq}, \nonumber\\  
& &  = 
\left\{ \sum_{\ell \in \mathcal{I}(p)} \alpha_{\ell}(m) [N_{\ell}]_{pp}\right\}^{-1/2}
\sum_{\ell \in \mathcal{I}(p\wedge q)} \alpha_{\ell}(m) [N_{\ell}]_{pq} \\
&& 
~~\times \left\{ \sum_{\ell \in \mathcal{I}(q)} \alpha_{\ell}(m) [N_{\ell}]_{qq}
\right\}^{-1/2}. \label{sa}
\end{eqnarray*}
For $p \in \mathcal{C}$, $[A(m)]_{pq} = 0$ if $p, q$ are not adjacent
cell-centers and $[A(m)]_{pp} = 1$. Furthermore note that, 
$[A(m)]_{pq}$ is $m$-dependent only if either $p$
or $q \in \bigcup_{k=1, \ldots, s+1} \Gamma_{\Omega_k}$.


It is sufficient to study only the single island case because single
island expression for 
\begin{equation} \label{K_Decomp}
K(m) = m N_1 + N_{2}(m)
\end{equation}
can be simply generalized to
\begin{equation*}
K(m) = \sum_{k=1}^s m N_k + N_{s+1}(m).
\end{equation*}

\begin{table*}[htbp]
\caption{A 2D example showing the condition numbers and eigenvalues of 
the finite volume discretization matrix $K(m)$ and its diagonally scaled
version $A(m)$ in which the eigenvalues are sorted in ascending order.} 
\label{tbl:diagScalingEffect}
$$\begin{array}{llllllllll}
  \hline \\
  m &~~ \kappa(K(m)) & \lambda_1(K(m)) & \lambda_{61}(K(m)) & 
 \lambda_{62}(K(m)) &\lambda_{64}(K(m))  & \kappa(A(m)) & \lambda_1(A(m)) &
 \lambda_2(A(m)) & \lambda_{64}(A(m))  \\[1ex]
  \hline \hline \\
  10^0 &~~ 2.627\times10^1 &~~ 3.045\times10^{-1} &~~ 7.695\times10^{0} &~~
 7.848\times10^{0} &~~ 8.000\times10^{0} &~~  2.553\times10^1 &~~ 7.538\times10^{-2} &~~ 1.800\times10^{-1}
 &~~ 1.925\times10^{0} \\[1ex]
  10^2 &~~ 1.248\times10^3 &~~ 3.237\times10^{-1} &~~ 7.995\times10^{0} &~~
 2.040\times10^2 &~~   4.040\times10^2 &~~ 3.448\times10^2 &~~
 5.784\times10^{-3} &~~  1.362\times10^{-1} &~~ 1.994\times10^{0} \\[1ex]
  10^4 &~~ 1.235\times10^5 &~~ 3.240\times10^{-1} &~~ 8.000\times10^{0} &~~ 
2.000\times10^4 &~~   4.000\times10^4 &~~ 3.258\times10^4 &~~
 6.139\times10^{-5} &~~   1.346\times10^{-1} &~~ 1.999\times10^{0}  \\[1ex]
  10^6 &~~ 1.235\times10^7 &~~ 3.240\times10^{-1} &~~ 8.000\times10^{0} &~~ 
2.000\times10^6 &~~  4.000\times10^6 &~~ 3.256\times10^6 &~~ 
6.142\times10^{-7} &~~   1.346\times10^{-1} &~~  2.000\times10^{0}  \\[1ex]
  10^8 &~~ 1.235\times10^9 &~~ 3.240\times10^{-1} &~~ 8.000\times10^{0} &~~
 2.000\times10^8 &~~   4.000\times10^8 &~~ 3.256\times10^8 &~~
 6.143\times10^{-9} &~~   1.346\times10^{-1} &~~ 2.000\times10^{0}  \\[1ex]
  10^{10} &~~1.235\times10^{11} &~~ 3.240\times10^{-1} &~~ 8.000\times10^{0} 
&~~ 2.000\times10^{10} &~~  4.000\times10^{10} &~~ 3.256\times10^{10}
 &~~ 6.143\times10^{-11} &~~   1.346\times10^{-1} &~~ 2.000\times10^{0} \\[1ex]
  \hline
\end{array}$$
\end{table*}

\subsection{Perturbation expansion analysis for the upper bound of
the smallest eigenvalue}

We devise a perturbation expansion analysis based on $m$, in order to
study the $m$-dependent spectral behavior of $A(m)$. Hence, only the
matrix entries $A(m)_{pq},~p \ne q$, that have $m$-dependence are
considered where cells $p$ and $q$ are adjacent.  Combining
\eqref{matrixEntryValues} and \eqref{K_Decomp}, one can deduce that
\begin{equation} \label{N_2Decomp0}
N_2(m) = N_2^\infty + \cO(m^{-1}).
\end{equation}
We will use \eqref{N_2Decomp0} in the below analysis.  For clarity, we
treat the perturbation expansion in full detail for the first case.

\noindent {\bf Case 1:} $p \in \Gamma_{\Omega_{1}}$ and $q \in \Omega_{2}$,
\begin{eqnarray*}
& &[A(m)]_{pq}  \\
& & = \left\{ m [N_{1}]_{pp} + [N_{2}(m)]_{pp} \right\}^{-1/2} 
[N_{2}(m)]_{pq} \left\{ [N_{2}(m)]_{qq} \right\}^{-1/2} \\ 
& & =  m^{-1/2} \left\{ [N_{1}]_{pp} +  m^{-1} [N_{2}(m)]_{pp} \right\}^{-1/2}
[N_{2}(m)]_{pq} \\
&& ~~\times \left\{ [N_{2}(m)]_{qq} \right\}^{-1/2} \\ 
& & = m^{-1/2} \left\{ ~m^{-1} [N_1]_{pp}^{-1/2} -1/2 [N_1]_{pp}^{-3/2}~
[N_2(m)]_{pp} \right.\\
&& ~~\left. +~\cO(m^{-2}) \right\} \left\{ [N_2(m)]_{pq} \right\} 
\left\{ [N_2(m)]_{qq} \right\}^{-1/2} \\ 
& & = m^{-1} \left\{ [N_1]_{pp}^{-1/2} + \cO(m^{-1}) \right\} 
\left\{ [N_2^\infty]_{pq} + \cO(m^{-1}) \right\}\\
&& ~~\times \left\{ m^{-1} [N_2^\infty]_{qq} + \cO(m^{-2}) \right\}^{-1/2}\\ 
& & = \cO(m^{-1/2})
\end{eqnarray*}

\noindent {\bf Case 2:} $p \in \Gamma_{\Omega_{1}}$ and 
$q \in \Gamma_{\Omega_{1}}$,
\begin{eqnarray*}
& &[A(m)]_{pq} \\
& & = \left\{ m [N_{1}]_{pp} + [N_{2}(m)]_{pp} \right\}^{-1/2} 
[mN_{1}]_{pq} \\
&& ~~\times \left\{ m [N_{1}]_{qq} + [N_{2}(m)]_{qq} \right\}^{-1/2}\\ 
& & = \left\{ [N_{1}]_{pp} + m^{-1}[N_{2}(m)]_{pp} \right\}^{-1/2} 
[N_{1}]_{pq} \\
&& ~~\times \left\{ [N_{1}]_{qq} + m^{-1} [N_{2}(m)]_{qq} \right\}^{-1/2}\\
& & = [N_{1}]_{pp}^{-1/2}[N_{1}]_{pq} [N_{1}]_{qq}^{-1/2} + \cO(m^{-1})
\end{eqnarray*} 

\noindent {\bf Case 3:} $p \in \Gamma_{\Omega_{1}}$ and $q \in I_{\Omega_{1}}$,
\begin{eqnarray*}
& &[A(m)]_{pq} \\
& & = \left\{ m [N_{1}]_{pp} + [N_{2}(m)]_{pp} \right\}^{-1/2} 
[mN_{1}]_{pq} \left\{ m [N_{1}]_{qq} \right\}^{-1/2} \\ 
& & = \left\{ [N_{1}]_{pp} + m^{-1}[N_{2}(m)]_{pp} \right\}^{-1/2} 
[N_{1}]_{pq} \left\{ [N_{1}]_{qq} \right\}^{-1/2}\\
& & = 
[N_{1}]_{pp}^{-1/2} [N_{1}]_{pq} [N_{1}]_{qq}^{-1/2} +
\cO(m^{-1})
\end{eqnarray*} 

\noindent {\bf Case 4:} $p \in \Gamma_{\Omega_{2}}$ and $q \in \Omega_{2}$,
\begin{eqnarray*}
& & [A(m)]_{pq} \\
& & = \left\{ [N_{2}(m)]_{pp} \right\}^{-1/2} 
[N_{2}(m)]_{pq} \left\{ [N_{2}(m)]_{qq} \right\}^{-1/2} \\ 
& & = \left\{ m^{-1}[N_{2}(m)]_{pp} \right\}^{-1/2} 
m^{-1} [N_{2}(m)]_{pq} \\
&  & ~~\times \left\{ m^{-1}[N_{2}(m)]_{qq} \right\}^{-1/2}\\ 
& & = 
\left\{ m^{-1} [N_2^\infty]_{pp} + \cO(m^{-2}) \right\}^{-1/2}
\left\{ m^{-1} [N_2^\infty]_{pq} + \cO(m^{-2}) \right\} \\
&& ~~\times \left\{ m^{-1} [N_2^\infty]_{qq} + \cO(m^{-2}) \right\}^{-1/2} \\
& & = [N_2^\infty]_{pp}^{-1/2} [N_{2}^\infty]_{pq} [N_{2}^\infty]^{-1/2}_{qq} + \cO(m^{-1}).
\end{eqnarray*} 

\begin{remark}
Numerically we observe that the smallest eigenvalue is $\cO(m^{-1});$
see Table~\ref{tbl:diagScalingEffect}.  In the above analysis, all the
cases except Case 1 give the same $\cO(m^{-1})$.  Case 1 result of
$\cO(m^{-1/2})$ estimate is an artifact of the perturbation expansion.
The same artifact has appeared in the FE analysis; see
\cite[equ. (5.14)]{GrHa:99}.
\end{remark}

We will use following modification of $N_2$
for our further analysis:
\begin{eqnarray}
\label{N2tilde}
[\tilde{N_2}]_{pq} =
\begin{cases}
0 & \,\text{ if } \, p \in \Gamma_{\Omega_1} \\ [N_2^\infty]_{pq} & \,\text{
  otherwise}. \\
\end{cases}
\end{eqnarray}
Consider the reduced version of ~\eqref{K_Decomp}
\begin{equation*}
\tilde{K}(m) = mN_1 + \tilde{N_2}, 
\end{equation*} 
and let $\tilde{A}(m)$ denote the diagonally scaled version of
$\tilde{K}(m)$. Then $m$-independent $\tilde{A}(m)$ has a single zero
\eig. Next, we proceed with the elementwise analysis of $A(m) -
\tilde{A}(m)$.

\noindent {\bf Case 1:} $p \in \Gamma_{\Omega_1}$ and $q \in \Omega_{2}$,
\begin{equation*}
  [A(m)]_{pq} = m^{-1/2}~[N_1]_{pp}^{-1/2}~[N_{2}^\infty]_{pq}~[N_{2}^\infty]^{-1/2}_{qq} + 
  \cO(m^{-3/2})
\end{equation*}
and from ~\eqref{N2tilde} we get
\begin{equation*}
[\tilde{A}(m)]_{pq} = 0. 
\end{equation*} 
Therefore,
\begin{equation*}
  [A(m)]_{pq} - [\tilde{A}(m)]_{pq} = \cO(m^{-1/2}).
\end{equation*}

\noindent {\bf Case 2:} $p \in \Gamma_{\Omega_1}$ and $q \in \Omega_{1}$,
\begin{equation*}  
  [A(m)]_{pq} = [N_{1}]_{pp}^{-1/2}~[N_{1}]_{pq}~[N_{1}]_{qq}^{-1/2} + \cO(m^{-1})
\end{equation*} 
and from ~\eqref{N2tilde} we get
\begin{equation*}
  [\tilde{A}(m)]_{pq} = [N_{1}]_{pp}^{-1/2}~[N_{1}]_{pq}[N_{1}]_{qq}^{-1/2}
\end{equation*} 
Thus,
\begin{equation*}
  [A(m)]_{pq} - [\tilde{A}(m)]_{pq} = \cO(m^{-1}).
\end{equation*}

\noindent {\bf Case 3:} $p \in \Gamma_{\Omega_{2}}$(nodes of $\Omega_2$),
\begin{equation*}
  [A(m)]_{pq} = [N_2^\infty]_{pp}^{-1/2}~[N_{2}^\infty]_{pq}~[N_{2}^\infty]^{-1/2}_{qq} + \oalp
\end{equation*}
and from ~\eqref{N2tilde}  we get
\begin{equation*} 
  [\tilde{A}(m)]_{pq} = [N_2^\infty]_{pp}^{-1/2}~[N_{2}^\infty]_{pq}~[N_{2}^\infty]^{-1/2}_{qq}.
\end{equation*}
Thus,
\begin{equation*}
  [A(m)]_{pq} - [\tilde{A}(m)]_{pq} = \cO(m^{-1}).
\end{equation*}

\noindent {\bf Case 4:} Otherwise,
\begin{equation*}
    [A(m)]_{pq} - [\tilde{A}(m)]_{pq} = 0.\\
\end{equation*}

Therefore, $\lambda_{\max} (A(m)-\tilde{A}(m)) = \cO(m^{-1/2})$.
\begin{lemma} \label{eig}
Let $G$ and $H$ be symmetric matrices of dimension $n \times n$. Then,
for  $k = 1, \dots, n$, the following holds:
\begin{equation*}
\lambda_k(G) + \lambda_{\min}(H) \leq \lambda_k(G+H) \leq
\lambda_k(G) +\lambda_{\max}(H). 
\end{equation*}
\end{lemma}
\begin{proof}
The result follows from Courant-Fischer minimax Theorem; see
\cite[Corollary 8.1.3]{GoVL89}. \qed
\end{proof} 

From Lemma ~\ref{eig}, we have
\begin{eqnarray*}
\lambda_{\min}(A(m)) 
& \leq & \lambda_{\min}(\tilde{A}(m)) + \lambda_{\max}(A(m) - \tilde{A}(m)) \\ 
&    = & \lambda_{\max}(A(m) - \tilde{A}(m)) \\
&    = & \cO(m^{-1/2})
\end{eqnarray*}
Moreover, we have for all $k \geq 1,$
\begin{eqnarray*}
\lambda_{k}(A(m)) 
& \geq & \lambda_{k}(\tilde{A}(m)) + \lambda_{\min}(A(m) - \tilde{A}(m)) \\ 
& \geq & \lambda_{k}(\tilde{A}(1)) -\cO(m^{-1}) \geq C
\end{eqnarray*} 
for some constant $C$ independent of $m$, asymptotically.  Thus,
$A(m)$ has a single \eig~ approaching to zero while the rest \eigs~is
bounded away from $0$. 
 
\subsection{Lower bound for the smallest \eig}
We aim to show the following lower bound for the smallest eigenvalue:
\begin{equation} \label{lowerBound}
\lambda_{\min}(A(m)) \geq C~m^{-1}.
\end{equation}
For that, we will establish the below main estimates in
the discussion to follow:
\begin{equation} \label{lowerBoundEst1}
x^TK(m)x \geq x^TK(1)x \geq m^{-1} x^TK(m)x.
\end{equation}
 
\begin{remark}
Establishing the estimate \eqref{lowerBoundEst1} is not as simple as
in the FE discretization case due to $m$-dependence of $N_2$ in
\eqref{K_Decomp}.  This requires further detailed matrix analysis.
\end{remark}

\subsubsection{$x^TK(m)x \geq x^TK(1)x$ estimate}

For the decomposition in \eqref{K_Decomp}, it is straightforward to
see that $m N_1 \geq N_1$ for $m \geq 1$. Hence, in order to establish
\eqref{lowerBoundEst1}, we concentrate on the following auxiliary estimate:
\begin{equation} \label{auxEstimateN_2}
x^T N_2(m)x \geq x^T N_2(1)x.
\end{equation}

First, note that $N_2(m)$ has positive diagonal and negative
off-diagonal entries. In addition, due to the discretization formula
\eqref{FVMdiscretization}, it has a row sum equal to zero with the
exception that the row sums corresponding to cell-centers that are 
adjacent to the boundary are positive.  Hence, $N_2(m)$ is a
diagonally dominant matrix. We further decompose $N_2(m)$ as
follows:
\begin{equation} \label{N_2Decomp}
N_2(m) = \bar{N}_2(m) + R_2,
\end{equation} 
where $\bar{N}_2(m)$ is a symmetric matrix with positive diagonal
entries and a row sum equal to zero, and $R_2 := N_2(m) - \bar{N}_2(m)$ is
the remainder matrix.

\begin{lemma} \label{gerschgorin}
Let $G$ be a symmetric matrix and have a row sum equal to zero. In
addition, let $[G]_{pp} \geq 0$, then $G$ is symmetric positive
semi-definite (SPSD).
\end{lemma}
\begin{proof}
Let $\lambda_{p}$ be the $p$-th \eig~of $G$. Using Gerschgorin's Theorem
with the fact that $G$ has a row sum equal to zero yields:
\begin{equation*}
|\lambda_p - [G]_{pp}| \leq \sum_{q \ne p}|[G]_{pq}| = [G]_{pp}.
\end{equation*}
The result follows from $0 \leq \lambda_p \leq 2[G]_{pp}.$ 
\qed
\end{proof}

By Lemma \ref{gerschgorin}, $\bar{N}_2(m)$ is immediately SPSD. 
In addition, $R_2$ is a diagonally dominant matrix with non-negative 
diagonal entries, again by  Lemma \ref{gerschgorin}, $R_2$ is SPSD.  
Now, we can conclude that $N_2(m)$ is SPSD.

From \eqref{matrixEntryValues}, one observes that $[K(m)]_{pp}$ is 
monotonically increasing in $m$.
This important property implies that
\begin{equation} \label{N2barMonotonic}
[\bar{N}_2(m)]_{pp} \geq [\bar{N}_2(1)]_{pp},~~m \geq 1.
\end{equation}
We need the following additional decomposition:
\begin{equation} \label{N_2barDecomp}
\bar{N}_2(m) = \bar{N}_2(1) + \bar{R}_2(m).
\end{equation}
Combining \eqref{N2barMonotonic} and \eqref{N_2barDecomp}, we obtain
$[\bar{R}_2(m)]_{pp} \geq 0$.  Hence, Lemma~\ref{gerschgorin} implies
that $\bar{R}_2(m)$ is SPSD.  Using this, we now arrive at the auxiliary
estimate \eqref{auxEstimateN_2} in the following:
\begin{eqnarray*}
x^TN_2(m)x & = & x^T\bar{N}_2(1)x +x^T \bar{R}_2(m) x + x^TR_2x \\
& \geq &  x^T\bar{N}_2(1)x + x^TR_2x \\
& = &  x^TN_2(1)x.
\end{eqnarray*}
Consequently,
\begin{eqnarray*}
mx^TK(m)x & = & m x^TN_1x + x^TN_2(m)x \\
& \geq & x^TN_1x + x^TN_2(m)x \\ 
&\geq & x^TN_1x + x^TN_2(1)x  \\ 
& = & x^TK(1)x.
\end{eqnarray*}

\subsubsection{$mx^TK(1)x \geq x^TK(m)x$ estimate}

Combining \eqref{K_Decomp} and \eqref{N_2Decomp},  we obtain
\begin{equation*}
K(m) = mN_1+ \bar{N}_2(m) + R_2,
\end{equation*} where $N_1$ and $R_2$ are SPSD and independent of $m$,
which yields the following for $m\geq 1$:
\begin{eqnarray*}
mx^TN_1x & \geq & x^TN_1x,\\
mx^TR_2x & \geq & x^TR_2x,
\end{eqnarray*}
Thus, it is sufficient to establish the auxiliary estimate
\begin{equation} \label{auxEstimateN_2Bar}
m x^T \bar{N}_2(1)x \geq x^T \bar{N}_2(m)x,
\end{equation}
in order to establish \eqref{lowerBoundEst1}.

By using Remark ~\ref{matrixEntryValues}, one can also show that 
\begin{equation} \label{mN2barMonotonic}
m [\bar{N}_2(1)]_{pp} \geq [\bar{N}_2(m)]_{pp}.
\end{equation}
We will use the following decomposition:
\begin{equation} \label{mN2Decomp}
m \bar{N}_2(1) = \bar{N}_2(m) +\widehat{R}_2(m),
\end{equation}
where $\widehat{R}_2(m)$ is the SPSD remainder matrix.
Combining \eqref{mN2barMonotonic} and \eqref{mN2Decomp},
we obtain $[\widehat{R}_2(m)]_{pp} \geq 0$, which leads us to the auxiliary
estimate \eqref{auxEstimateN_2Bar}.

Hence,
\begin{eqnarray*} 
mx^TK(1)x & = & mx^TN_1x + mx^TN_2(1)x \\
& \geq & mx^T N_1 x + x^TN_2(m)x\\
& = & x^TK(m)x.
\end{eqnarray*}

In conclusion, we have obtained the two main estimates in
\eqref{lowerBoundEst1}. These yield the similar estimates
for $\textrm{diag}~K(m)$. 
\begin{eqnarray}
x^T \textrm{diag}~K(m)x & \geq & x^T \textrm{diag}~K(1)x \nonumber \\
& \geq & m^{-1} x^T \textrm{diag}~K(m)x \label{lastStep2}
\end{eqnarray}
From \eqref{lowerBoundEst1} and \eqref{lastStep2}, and letting 
$C_1:= \lambda_{\min}(A(1))$, we get:
\begin{eqnarray*} 
x^TK(m)x & \geq & x^TK(1)x \geq C_1 x^T \textrm{diag}~K(1)x \\
& \geq & C_1~m^{-1} x^T \textrm{diag}~K(m)x,
\end{eqnarray*}
yielding
\begin{equation*}
C_1~m^{-1} \leq \lambda_{\min}(A(m)) \leq C_2~m^{-1/2}.
\end{equation*}

\begin{remark}
The merit of diagonal scaling becomes apparent after studying the
spectral behavior of $A(m)$.  We observe that the number of the
\eigs~of $A$ of $\cO(m^{-1})$ depends only on the number of isolated
islands. On the other hand, the number of the \eigs~of $K(m)$ of
$\cO(m)$ depends on the number of DOF of the islands,
asymptotically. In the example we report in Table
\ref{tbl:diagScalingEffect}, $K(m)$ has $3$ \eigs~approaching to
$\infty$ and $61$ bounded \eigs, whereas, $A(m)$ has only one
\eig~approaching to zero. The reduction in the number of $m$-dependent
eigenvalues is a desirable feature for fast convergence of Krylov
subspace solvers.
\end{remark}

\section{Singular perturbation analysis on matrix entries}
\label{mainResultsSPA}

\subsection{Preliminaries on matrix properties}

The discretization formula \eqref{FVMdiscretization} with the harmonic
means \eqref{harmonicMeans} gives rise to the matrix entries given in
\eqref{matrixEntryValues}. Due to harmonic means, the $m$-dependence
of the matrix entries corresponding to $HL$ and $LL$ couplings is
asymptotically eliminated.  As a result, the submatrices $K_{HL}(m)$
and $K_{LL}(m)$ do not have $m$-dependence asymptotically:
\begin{eqnarray}
  K_{HL}(m) & = & K_{HL}^\infty + \cO(m^{-1}), \nonumber \\
  K_{LH}(m) & = & K_{HL}^{\infty^T} + \cO(m^{-1}), \label{mDependenceElim}\\
  K_{LL}(m) & = & K_{LL}^\infty + \cO(m^{-1}) \nonumber.
\end{eqnarray}
The analysis in this article mainly relies on this
crucial property. 

To analyze the $m$-robustness of preconditioners based on
\eqref{eq:exact}, we need to analyze the asymptotic behaviour of the
block components $K_{HH}(m)^{-1}$, $S(m)^{-1}$, and\\
$K_{LH}(m) K_{HH}(m)^{-1}$ as $m \rightarrow \infty$.  This is the purpose
of Lemma \ref{lemma:Main} below.  To prepare for this, we further
decompose DOF associated with $\overline{\Omega}_H$
into a set of interior DOF associated with index $I$ and boundary
DOF with index $\Gamma$. This leads to the following further
block representation of
\begin{equation} \label{K_HH_blockwise} 
   K_{HH}(m) \ = \
            \left[ \begin{array}{cc}
               K_{II}(m) & K_{I \Gamma}(m) \\
              K_{\Gamma I}(m) & K_{\Gamma \Gamma}(m)
                   \end{array} \right].
\end{equation}
By using \eqref{matrixEntryValues}, we can write a more refined
expression for the block $K_{\Gamma \Gamma}(m)$:
\begin{equation*}
K_{\Gamma \Gamma}(m) = K^{(H)}_{\Gamma \Gamma}(m) + 
\hm D^{(L)}_{\Gamma\Gamma},
\end{equation*}
with
\begin{equation} \label{diagonalDelta}
D^{(L)}_{\Gamma\Gamma} := 
\textrm{diag}(0, \ldots,0, 1, \ldots, 1, 2, \ldots, ,2),
\end{equation}
where the number of $0$, $1$, and $2$ entries is equal to the cardinality
of interior, non-corner interface, and corner interface cell-centers,
respectively. Thus, we can characterize the Neumann matrix
$\mathcal{N}_{HH}$ on $\Omega_H$ as described in (\ref{eq:KHH}) as
follows:
\begin{eqnarray}
  K_{HH}(m) & = &
  m \mathcal{N}_{HH} \
  + \ \Delta(m) \label{eq:Neumann} \,,\\
  \Delta(m) & = & \hm
  \left[ \begin{array}{cc}
      0 & 0 \\
      0 & D^{(L)}_{\Gamma \Gamma}
    \end{array} \right] , \label{deltaDecomp}\\
  & = & \Delta^\infty + \oalp \label{mDependenceDelta},
\end{eqnarray}
where $\Delta(m)$ is a diagonal matrix due to \eqref{diagonalDelta}.
$\mathcal{N}_{HH}$ is a SPSD matrix with a simple zero eigenvalue and
associated constant eigenvector. If $n_H$ denotes the number of DOF in
$\Omega_H$, a suitable normalized eigenvector is the constant vector
with entries $n_H^{-1/2}$, which we denote by $e_H$.  We further write
in block form as $e_H^T = ( e_I^T \ , \ e_{\Gamma}^T)$.

\begin{remark}
  The Neumann matrix in the FV discretization is completely different
  than that of the FE case.  However, as expected, the simple zero
  eigenvalue and the associated constant eigenvector are maintained
  because this is an inherent property of the underlying PDE. In
  addition, $\Delta$ in \eqref{eq:Neumann} in the FE discretization is
  not necessarily diagonal and does not have $m$-dependence.
\end{remark}

Finally we note that the off-diagonal blocks in \eqref{2x2blockSys}
have the decomposition:
\begin{equation}
  \label{K_HL_blockwise}
  K_{LH}(m) \ = \ \left[ \begin{array}{cc}
      0 & K_{L\Gamma}(m)
    \end{array} \right] \ = \ K_{HL}(m)^T.
\end{equation}

As we prepare for the proof of our main Lemma, we need to define
the following quantity which will also be used for subdomain
deflation in \eqref{deflationProj}:
\begin{equation} \label{etaExpression}
\eta(m) := e_H^T K_{HH}(m)e_H.
\end{equation}
$\eta(m) > 0$ because $K_{HH}(m)$ is SPD as a diagonal subblock of
$K(m)$. Moreover, combining \eqref{eq:Neumann} and \eqref{deltaDecomp}, 
one can reduce the expression in \eqref{etaExpression} to
\begin{equation} \label{etaExpressionReduced}
\eta(m) = \hm \, e_{\Gamma}^{T}  D_{\Gamma \Gamma}^{(L)} e_{\Gamma}.
\end{equation}
In fact, \eqref{etaExpressionReduced} can be expressed explicitly by:
\begin{equation} \label{etaExpressionExplicit}
\eta(m) = \hm~\frac{n_{H,nc} + n_{H,c}}{n_H},
\end{equation}
where 
\begin{eqnarray*}
n_{H_,nc} & := & \#\{\textrm{non-corner interface cell-centers}\}\\
n_{H_,c}  & := & \#\{\textrm{corner interface cell-centers}\}.
\end{eqnarray*}
Finally, \eqref{etaExpressionExplicit} delivers the asymptotic
convergence of expression for $\eta(m)$:
\begin{equation} \label{etaExpressionAsymp}
\eta(m) = \eta + \cO(m^{-1}),
\end{equation}

\subsection{The main results on the preconditioner}

\begin{lemma} \label{lemma:Main} 
The asymptotic behaviour of the submatrices in \eqref{eq:exact} 
is described by the following:
\begin{itemize}
  \item[(i)] $K_{HH}(m)^{-1} = e_{H}\eta^{-1} e_{H}^{T} +
    \mathcal{O}(m^{-1})$,
  \item[(ii)] $S(m) = K_{LL}^{\infty}-
    (K_{L\Gamma}^{\infty} e_{\Gamma})\eta^{-1}(e_{\Gamma}^{T}K_{\Gamma L}^{\infty}) +
    \mathcal{O}(m^{-1})$,
  \item[(iii)] $K_{LH}(m) K_{HH}(m)^{-1} = (K_{L \Gamma}^{\infty}e_{\Gamma})
    \eta^{-1} e_{H}^{T} + \mathcal{O}(m^{-1})$.
\end{itemize}
\end{lemma}

\begin{proof}
  Since $\cN_{HH}$ is symmetric positive semidefinite we have the
  eigenvalue decomposition:
\begin{equation} \label{spectralDecomp1}
Z^T \cN_{HH} Z =
\text{diag}(\lambda_1, \ldots, \lambda_{n_H-1}, 0),
\end{equation}
where $\{ \lambda_i :\ i = 1, \ldots, n_H \}$ is a non-increasing
sequence of eigenvalues of $\cN_{HH}$ and $Z$ is orthogonal.  Since,
the eigenvector corresponding to the zero eigenvalue is constant, we
can write $Z = \left[\tilde{Z} \ | \ e_H \right]$ and so, using
\eqref{eq:Neumann}, we have

\begin{eqnarray}
\label{zaz}
& & Z^{T}K_{HH}(m)Z \nonumber \\ 
& & = \left[
\begin{matrix}
  m~\textrm{diag} (\lambda_{1}, \ldots, \lambda_{n_{H-1}}) +
  \tilde{Z}^{T}\Delta(m) \tilde{Z} & ~\tilde{Z}^{T}\Delta(m) e_{H}
  \\ e_{H}^{T} \Delta(m) \tilde{Z} & e_{H}^{T} \Delta(m) e_{H} \\
\end{matrix} 
\right] \nonumber \\ 
& & =: \left[
\begin{matrix}
\Lamalp & \delt(m) \\ \deltt(m) & \eta(m) \label{spectralDecomp1b} \\
\end{matrix}
\right].
\end{eqnarray}
To find the limiting form of $K_{HH}(m)^{-1}$ note that
\begin{eqnarray*}
\tilde{\Lambda}(m) & = & 
m~\textrm{diag}(\lambda_{1}, \ldots, \lambda_{n_{H-1}}) + 
\tilde{Z}^{T} \Delta(m) \tilde{Z} \\
&=&
m~\textrm{diag}(\lambda_{1}, \ldots, \lambda_{n_{H-1}})\\
&& \times \left( \tilde{I} +
m^{-1}~\textrm{diag}(\lambda_{1}^{-1}, \ldots,
\lambda_{n_{H-1}}^{-1})\tilde{Z}^{T} \Delta (m) \tilde{Z} \right).
\end{eqnarray*}

Now, let $C_\lambda :=\max_{i<n_H} \lambda_i^{-1}$ and
$C_{\Delta}$ be the constant in \eqref{mDependenceDelta}, i.e.,
$ \|\Delta(m)\|_2 \leq \|\Delta^\infty\|_2 +C_\Delta m^{-1}.$  Then,
for sufficiently large $m$,
\begin{eqnarray}
  \|\tilde{\Lambda}(m)^{-1}\|_2 \nonumber & \leq &
  \frac{m^{-1} \, C_\lambda}
  {1 - m^{-1} \, C_\lambda \, \|\tilde{Z}^T \Delta(m) \tilde{Z}\|_2} \nonumber \\
  & \leq & 
  \frac{m^{-1} \, C_\lambda}
  {C_{\tilde{\Lambda}}} \label{mainPerturbResult}
\end{eqnarray}
where 
\begin{eqnarray}
  C_{\tilde{\Lambda}} & := &  
  1 - m^{-1} \, C_\lambda \, \|\tilde{Z}^T\|_2 \|\Delta^\infty\|_2 
\|\tilde{Z}\|_2  
  \nonumber \\
  && - m^{-2} \, C_{\Delta} \, C_\lambda \, \|\tilde{Z}^T\|_2 \|\tilde{Z}\|_2
  \nonumber \\
  & = & 1 + \cO(m^{-1}) \label{auxmainPerturbResult}.
\end{eqnarray}
Hence, combining \eqref{mainPerturbResult} and \eqref{auxmainPerturbResult},
we obtain:
\begin{equation} \label{result1_lemma}
   \tilde{\Lambda}(m)^{-1} =  \cO(m^{-1}).
\end{equation}
We proceed with the following inversion:
\begin{equation*} \label{inversion1} \left[ \begin{array}{cc}
      \tilde{\Lambda}(m) & \delt(m) \\
      \delt(m)^T & \eta(m)
\end{array} \right]^{-1} = U(m)~V(m)~U(m)^T,
\end{equation*}
where
\begin{eqnarray*}
  U(m) & := &
  \left[ \begin{array}{cc}
      \tilde{I} & -\tilde{\Lambda}(m)^{-1} \delt(m) \\
      0^T & 1
    \end{array} \right],\\
  V(m) & := &
  \left[ \begin{array}{cc}
      \tilde{\Lambda}(m)^{-1} & 0 \\ 
      0^T & \left(\eta(m) - \delt(m)^T \tilde{\Lambda}(m)^{-1} \delt(m)
      \right)^{-1}
\end{array} \right].
\end{eqnarray*}

Using \eqref{mDependenceDelta}, one gets
\begin{equation} \label{result2_lemma} \delt(m) = \tilde{Z}^T \Delta^\infty
  e_H + \cO(m^{-1}).
\end{equation}
Then, \eqref{etaExpressionAsymp}, \eqref{result1_lemma}, 
and \eqref{result2_lemma} imply that
\begin{eqnarray*} \label{result3_lemma}
  \eta(m)^{-1} & = & \eta^{-1} + \cO(m^{-1}), \label{result4_lemma}\\
  U(m) & = & I + \cO(m^{-1}), \\
  V(m) & = & \left[ \begin{array}{cc} O & 0 \\ 0^T & \eta^{-1}
\end{array} \right]  + \cO(m^{-1}).
\end{eqnarray*}
Combining the above results, we arrive at
\begin{equation}
  \label{inversion_limit}
  \left[ \begin{array}{cc}
      \tilde{\Lambda}(m) & \delt(m) \\
      \delt(m)^T & \eta(m)
    \end{array} \right]^{-1}
  \ = \
  \left[ \begin{array}{cc}
      O & 0 \\ 0^T & \eta^{-1}
    \end{array} \right]
  \ + \ \calg{O}(m^{-1})\ ,
\end{equation}
and, by \eqref{spectralDecomp1b}, we have
\begin{eqnarray*}
  \label{eq:largest}
  K_{HH}(m)^{-1} \ & = & 
  \ Z \left[ \begin{array}{cc}
      O & 0 \\ 0^T & \eta^{-1}
    \end{array} \right] Z^T \ + \ \calg{O}(m^{-1}) \ \\
  & = & \ e_H \left(
    e_{\Gamma}^T ~2\, D^{(L)}_{\Gamma \Gamma}
    e_{\Gamma} \right)^{-1} e_H^T  \ + \
  \calg{O}(m^{-1}) \ ,
\end{eqnarray*}
which proves part (i) of the Lemma.

Parts (ii) and (iii) follow from simple substitution, using
\eqref{SchurComplement1} and \eqref{K_HL_blockwise}.  \qed
\end{proof}

We use the following limiting forms in the definition of the proposed
preconditioner \eqref{MainPrec1}: 
\begin{eqnarray}
K_{HH}^{\infty^\dagger} & := & e_H \eta^{-1} e_H^T,  \label{limitingForm1}\\
Q_{LH}^\infty & := & K_{LH}^\infty K_{HH}^{\infty^{\dagger}}, \nonumber \\
S^\infty & := & 
K_{LL}^\infty - K_{LH}^\infty K_{HH}^{\infty^{\dagger}}  K_{HL}^\infty
\label{limitingForm2}.
\end{eqnarray}

Based on the above perturbation analysis we propose the following
preconditioner:
\begin{eqnarray}
  B_{AGKS}(m) \ &:=& \
  { \left[ \begin{array}{cc}  I_{HH} & - Q_{LH}^{\infty^T} \\
        0 & I_{LL}\end{array}\right]}
  { \left[ \begin{array}{cc}    K_{HH}(m)^{-1}& 0 \\
        0&  S^{\infty^{-1}} \end{array}\right]} \nonumber \\ 
  && \times { \left[ \begin{array}{cc}  I_{HH} & 0 \\    - Q_{LH}^{\infty} &
        I_{LL} \end{array}\right]}. \label{MainPrec1}
\end{eqnarray}

The following theorem shows that $B_{AGKS}$ is an effective preconditioner
for $m \gg 1$.
\begin{theorem}
\label{th:prec_robust}
For $m$ sufficiently large we have
\[
\sigma(B_{AGKS}(m)~K(m)) \ \subset \ [1-cm^{-1/2},1+cm^{-1/2}]
\]
for some constant $c$ independent of $m$, and therefore
\[
\kappa(B_{AGKS}(m)~K(m)) \ = \ 1 \ + \ \cO(m^{-1/2}).
\]
\end{theorem}

\begin{proof}
  The $m$-dependence of $K_{HL}(m), K_{LH}(m)$, and \\ $K_{LL}(m)$ are
  eliminated asymptotically as in \eqref{mDependenceElim}. Furthermore,
  $m$-dependence of $\eta(m)$ is also eliminated as a result of
  \eqref{mDependenceDelta}.  Therefore, the result follows by slightly
  modifying the proof, i.e., by incorporating the $m$- dependencies of
  the above entities, given in Aksoylu et al.~\cite{AGKS:2007} for the
  FE case. \qed
\end{proof}

\section{Qualitative nature of the solution of the high-contrast
  diffusion equation and decoupling}
\label{sec:qualitativeNature}
\begin{figure}[tbp]
  \centering{
  \includegraphics[width=1.4in]{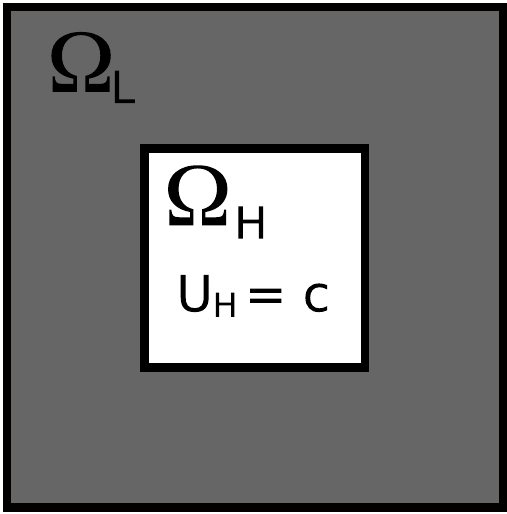}\hspace{.5cm}
  \includegraphics[width=1.6in]{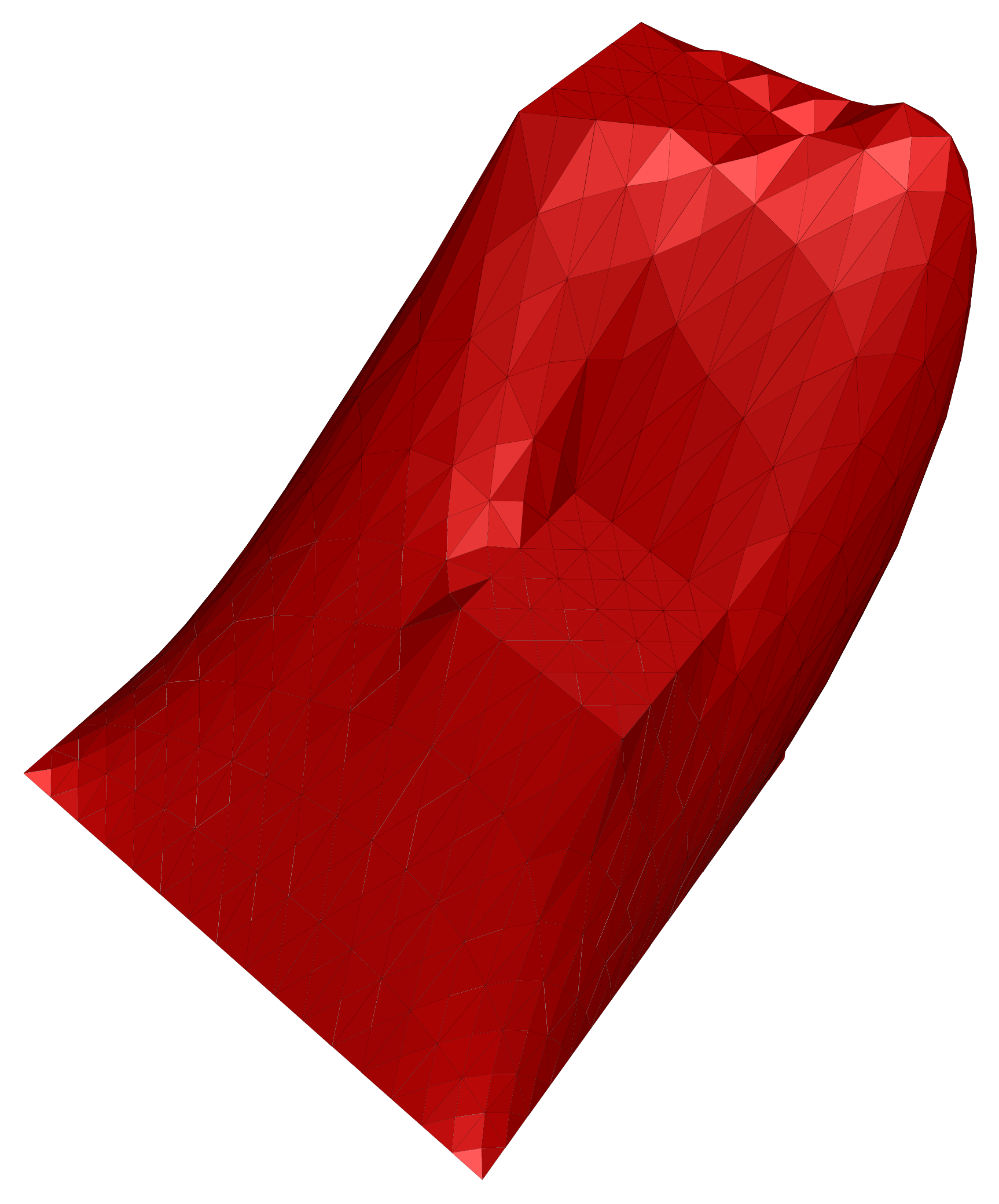}
}
\caption{(Left) The matrix in (\ref{ASlimitPlus1})
    corresponds to a homogeneous Dirichlet for the Laplacian on
    $\Omega$ under the constraint that the solution is constant on
    $\Omega_H$.  (Right) $x_{H_i}(m_i) = c_{H_i} +
    \mathcal{O}(m_i^{-1}),~i=1,2$ where $\Omega_{H_1}$ and
    $\Omega_{H_2}$ correspond to square and triangle shaped highly-diffusive
    islands, respectively with $m_i=10^6$.
  \label{fig:hiPermIslands}}
\end{figure}

We advocate the usage of SPA because it is a very effective tool in
gaining qualitative insight about the asymptotic behavior of the
solution of the underlying PDE.  Through SPA, in
Lemma~\ref{lemma:Main}, we were able to fully reveal the asymptotic
behaviour of the submatrices of $K$ in \eqref{eq:exact}. This
information leads to a characterization of the limit of the underlying
discretized inverse operator and we studied this in more detail
in~\cite{AkBe2009}.  We now prove that \emph{asymptotically, the
  solution over the highly-diffusive island goes to a constant
  vector}. This is probably the most fundamental qualitative feature
of the solution of the high-contrast diffusion equation.

\begin{lemma}
Let $e_H$ be the constant vector. Then, 
\begin{equation} \label{x_HConst}
x_H(m) = c_H~e_H  \ + \ \mathcal{O}(m^{-1}),
\end{equation}
where $c_H$ is determined by the solution in the lowly-diffusive region.
\end{lemma}

\begin{proof}
We prove the result by providing an explicit quantification of
the limiting process based on Lemma~\ref{lemma:Main}:
\begin{equation*}
\begin{array}{lllll}
x_L(m) & = & S^{-1}(m)~
  \{b_L - K_{LH}(m) K_{HH}^{-1}(m) b_H \} &&\\ 
  & = & S^{\infty^{-1}} \{b_L - 
  K_{LH}^\infty K_{HH}^{\infty^{\dagger}} b_H \} + \mathcal{O}(m^{-1})\\
  & =: & x_L^\infty + \mathcal{O}(m^{-1}),\\
  x_H(m) & = & 
  K_{HH}^{-1}(m)~ \{b_H - K_{HL}(m) x_L(m) \} &&\\ 
  & = & K_{HH}^{\infty^\dagger} 
  \{ b_H - K_{HL}^\infty x_L^\infty \} +
  \mathcal{O}(m^{-1})\\
  & =: & c_H~e_H  \ + \ \mathcal{O}(m^{-1}).
\end{array}
\end{equation*}
\qed
\end{proof}

SPA helps to reveal a further qualitative property, namely, the
decoupling phenomenon. We show how the original algebraic solution
strategy decouples into two algebraic problems of different nature and
the decoupling is indeed an implication of \eqref{x_HConst}.  This
observation can be a promising research avenue when it is carried over
to different classes of PDEs and discretizations. We will pursue this
in our future research.

In order to show the decoupling, let us start by noting that
$S^\infty$ in Lemma~\ref{lemma:Main} can also be interpreted as the
Schur complement of 
$c^2~e_{\Gamma}^T K^{(L)}_{\Gamma \Gamma} e_{\Gamma}$ in the matrix
\begin{equation*}
KK_{LL}^{\infty} \:= \ \left[ \begin{array}{cc}
c^2~e_{\Gamma}^T K^{(L)}_{\Gamma \Gamma} e_{\Gamma}  &
c~e_{\Gamma}^T K_{\Gamma L} \\
c~K_{L \Gamma} e_{\Gamma} & K_{LL}
\end{array} \right] \ ,
\end{equation*}
for any nonzero value of $c$. In particular, if we choose $c :=
n_H^{1/2}$, then $c e_{\Gamma} = 1_{\Gamma}$, the
vector of all ones on $\Gamma$ and, using also \eqref{eq:Neumann}, we
have
\begin{eqnarray*} 
KK_{LL}^{\infty} & := & \left[ \begin{array}{cc}
1_{\Gamma}^T K^{(L)}_{\Gamma \Gamma} 1_{\Gamma}  &
1_{\Gamma}^T K_{\Gamma L} \\
K_{L \Gamma} 1_{\Gamma} & K_{LL}
\end{array} \right] \nonumber \\ 
& = & 
\left[ \begin{array}{cc}
1_{H}^T K_{HH}(1) 1_{H}  &
1_{H}^T K_{HL} \\
K_{LH} 1_{H} & K_{LL}
\end{array} \right] \label{ASlimitPlus1}.
\end{eqnarray*}
This is the stiffness matrix for a pure Dirichlet problem for the
Laplacian on all of $\Omega$ with the additional constraint that the
solution is constant on $\Omega_H$. See Figure~\ref{fig:hiPermIslands}.
Thus, when $m \gg 1$, the original problem \emph{decouples almost
entirely into a (regularized) Neumann problem}
(i.e. $K_{HH}(m)$) for the Laplacian on $\Omega_H$ (scaled
by $m$) and a Dirichlet problem (i.e. $KK_{LL}^\infty$) for
the Laplacian on all of $\Omega$, but under the additional constraint
that the solution is constant on $\Omega_H$. 

Next, we show an additional decoupling. This time the preconditioner
decouples into a block diagonal matrix with the help of a deflation
method.

\section{Implementation aspects and the related deflation method}
\label{sec:subdomainDefl}

The fact that $\cN_{HH}$ has a simple zero eigenvalue (with the
corresponding constant eigenvector $e_H$) and $\cN_{HH}$ is of co-rank
$1$ imply that $K_{HH}(m)$ has a single eigenvalue of $\cO(1)$ and
$n_H - 1$ eigenvalues of $\cO(m)$.  For sufficiently large $m$, $e_H$
can be taken as an approximate eigenvector corresponding to
$\cN_{HH}$'s single smallest eigenvalue. Therefore, in the light of 
\eqref{eq:exact}, 
\begin{equation*}
e^T := [e_H^T,0^T],
\end{equation*}
becomes an approximate eigenvector corresponding to the smallest
eigenvalue of the decoupled matrix:
\begin{equation}
\left[ \begin{matrix} \label{decoupledK}
K_{HH}(m) & 0 \\ 0 & S(m)
\end{matrix} \right].
\end{equation}

In order to eliminate the negative effect of the smallest eigenvalue,
we utilize a deflation method under the constraint that it gives the
decoupling as in \eqref{decoupledK}. If such decoupling occurs, it
would provide a large computational advantage. We utilize a deflation
method, known as 
\emph{subdomain deflation}~\cite{Nicolaides,VuSeMe:1999,Vuik1999,Vuik01}
constructed by the following
$K$-orthogonal projection onto the subspace orthogonal to $e$, provides
the desired decoupling.
\begin{equation} \label{deflationProj}
{\cP^T} \ := \ I -  e \, \eta(m)^{-1} \, e^T \, K \,.
\end{equation}

We apply $B_{AGKS}$ within a conjugate gradient algorithm for
the deflated system
\begin{equation} \label{deflated}
\cP K x^\perp \ = \ \cP b,
\end{equation}
where $x^\perp := \cP^T x$ is the projected solution.
The component of $x$ in the direction of $e$ is then simply given by
\begin{equation*}
x - x^\perp \ = \ (I - \cP^T) x = \ e \, \eta(m)^{-1} \, e^T \, b \ .
\end{equation*}

When we rewrite \eqref{MainPrec1} as 
\begin{equation*}
B_{AGKS}(m) = L^T  \, D(m) \, L,
\end{equation*}
where 
\begin{eqnarray*}
L & := & 
\left[ 
\begin{array}{cc}  I_{HH} & 0 \\ - Q_{LH}^{\infty} & I_{LL} \end{array}
\right], \\
D(m) & := & 
\left[ 
\begin{array}{cc} K_{HH}(m)^{-1}& 0 \\ 0&  S^{\infty^{-1}} \end{array}
\right].
\end{eqnarray*}
By observing the following interesting property
\begin{equation} \label{interestingPropLP}
L \, \cP = \cP,
\end{equation}
the system in \eqref{deflated} after preconditioned by $B_{AGKS}(m)$
turns into:
\begin{equation*}
B_{AGKS}(m) \, \left(\cP K x^\perp \right) = B_{AGKS}(m) \, \left(\cP b \right).
\end{equation*}
Then, it reduces to the following block diagonal system:
\begin{equation*}
D(m)~\cP\,K\,x^\perp = D(m)~\cP \,b.
\end{equation*}

\begin{remark}
The fact that $e$ becomes an approximate eigenvector corresponding to
the smallest eigenvalue of the decoupled matrix in \eqref{decoupledK}
is not necessarily true for the original matrix $K(m)$ because it is
not block diagonal.  Therefore, utilizing the above deflation method
in the direction of $e$ will not necessarily provide any further
robustness for the underlying preconditioner if that preconditioner
uses off-diagonal blocks of $K(m)$ as well.  Multigrid method is one
such preconditioner.  In fact, numerically we observed that
introducing deflation did not improve the multigrid convergence rate
at all. If one still wants to improve the convergence rate by the help
of deflation, then the approximate eigenvector corresponding to the
smallest eigenvalue must be computed in an alternative way rather than
the simple usage of $e$.  Consequently, we can say that $B_{AGKS}(m)$,
by its design, naturally goes with subdomain deflation.  The
incorporation of subdomain deflation not only brings robustness with
respect to the smallest eigenvalue, but also provides huge
computational savings due to reduction to a block diagonal
system. This is a significant computational advantage $B_{AGKS}(m)$ offers.
\end{remark}

By exploiting the fact that $S^\infty$ in \eqref{limitingForm2} is
only a low-rank perturbation of $K^\infty$, we can build robust
preconditioners for $S^\infty$ in \eqref{MainPrec1} via standard
multigrid preconditioners.  Combining \eqref{limitingForm1} and
\eqref{limitingForm2}, we arrive at
\begin{equation*}
S^\infty = K_{LL}^\infty - v \eta^{-1}v^T,
\end{equation*}
where $v := K_{LH}^\infty e_H$. If $M_{LL}$ denotes a
standard multigrid V-cycle for $K_{LL}$, we can construct an
efficient and robust preconditioner $\tilde{S}^{-1}$ for $S^\infty$
using the Sherman-Morrison-Woodbury formula, i.e.
\begin{equation} \label{Sherman-Morrison}
\tilde{S}^{-1} \ := \ M_{LL} \ + \ M_{LL} v ~(1 - \eta)^{-1} \, v^T M_{LL}.
\end{equation}
Note also that we can precompute and store $M_{LL} v$ during the setup
phase.  This means we only need to apply the multigrid V-cycle
$M_{LL}$ once per iteration.  Therefore, the following practical
version of preconditioner \eqref{MainPrec1} is used in the
implementation:
\begin{eqnarray}
\tilde{B}_{AGKS} & := &
{ \left[ \begin{array}{cc}  I_{HH} & - Q_{LH}^{\infty^T} \\
0 & I_{LL}\end{array}\right]}
{ \left[ \begin{array}{cc}    M_{HH} & 0 \\
0&  \tilde{S}^{-1} \end{array}\right]} \nonumber \\
&& \times 
{ \left[ \begin{array}{cc}  I_{HH} & 0 \\    - Q_{LH}^\infty & I_{LL}
\end{array}\right]} \label{practicalAGKS}.
\end{eqnarray}

\begin{table*}
\caption{Preconditioner = CCMG, Prolongation = Wesseling-Khalil, Smoother = ILU}
\label{CWI}      
$$\begin{array}{rrrrrrrrrrr}
  \hline \\
  h \backslash m & 10^0 & 10^1 & 10^2 & 10^3 & 10^4 & 10^5 & 10^6 & 10^7 &
  10^8 & 10^9 \\[1ex]
  \hline\hline \\
  1/8 & \mathbf{10}, 0.025 & \mathbf{29}, 0.472 & \mathbf{30}, 0.500 & 
  \mathbf{35}, 0.528 & \mathbf{36}, 0.559 & \mathbf{40}, 0.593 & \mathbf{44}, 
  0.622 & \mathbf{52}, 0.669 & \mathbf{52}, 0.669 & \mathbf{58}, 0.697 \\[1ex]
  1/16 & \mathbf{3}, 0.000 & \mathbf{3}, 0.000 & \mathbf{3}, 0.001 &
  \mathbf{6}, 0.024 & \mathbf{15}, 0.236 & \mathbf{60}^+, 0.934 & \mathbf{60}^+, 
  0.934 & \mathbf{60}^+, 0.934 & \mathbf{60}^+, 0.934 & \mathbf{60}^+, 0.934 
  \\[1ex]
  1/32 & \mathbf{3}, 0.000 & \mathbf{3}, 0.000 & \mathbf{4}, 0.004 & 
  \mathbf{9}, 0.074 & \mathbf{36}, 0.558 & \mathbf{60}^+, 0.891 & \mathbf{60}^+,
  0.899 & \mathbf{60}^+, 0.899 &  \mathbf{60}^+, 0.899 & 
  \mathbf{60}^+, 0.899 \\[1ex]
  1/64 &~~~ \mathbf{3}, 0.000 &~~~  \mathbf{4}, 0.001 &~~~  \mathbf{5}, 0.015
  &~~~  \mathbf{9}, 0.084 &~~~  \mathbf{15}, 0.239 &~~~  \mathbf{20}, 0.315
  &~~~  \mathbf{15}, 0.246 &~~~  \mathbf{9}, 0.093 &~~~  \mathbf{8}, 
  0.067 &~~~  \mathbf{8}, 0.032 \\[1ex]
  \hline
\end{array}$$
\end{table*}

\begin{table*}
\caption{Preconditioner = CCMG, Prolongation = Wesseling-Khalil, Smoother = sGS}
\label{CWS}      
$$\begin{array}{rrrrrrrrrrr}
  \hline \\
  h \backslash m & 10^0 & 10^1 & 10^2 & 10^3 & 10^4 & 10^5 & 10^6 & 10^7 &
  10^8 & 10^9\\[1ex]
  \hline \hline \\
  1/8 &~~~ \mathbf{10}, 0.025 &~~~ \mathbf{29}, 0.472 &~~~ \mathbf{30}, 0.500 
  &~~~ \mathbf{35}, 0.528 &~~~ \mathbf{36}, 0.559 &~~~ \mathbf{40}, 0.593 &~~~
  \mathbf{44}, 0.622 &~~~ \mathbf{52}, 0.669 &~~~ \mathbf{52}, 0.669 &~~~
  \mathbf{58}, 0.697 \\[1ex]
  1/16 & \mathbf{12}, 0.166 & \mathbf{14}, 0.222 & \mathbf{19}, 0.310 &
  \mathbf{26}, 0.409 & \mathbf{38}, 0.535 & \mathbf{52}, 0.649 & \mathbf{60}^+,
  0.777 & \mathbf{60}^+, 0.843 & \mathbf{60}^+, 0.938 & \boldsymbol{\infty}, 
  1.070  \\[1ex]
  1/32 & \mathbf{14}, 0.195 & \mathbf{15}, 0.217 & \mathbf{18}, 0.294 & 
  \mathbf{25}, 0.432 & \mathbf{39}, 0.552 & \mathbf{58}, 0.698 & \mathbf{60}^+,
  0.917 & \boldsymbol{\infty}, 1.002 & \boldsymbol{\infty}, 1.080 &
  \boldsymbol{\infty}, 1.127 \\[1ex]
  1/64 & \mathbf{13}, 0.197 & \mathbf{14}, 0.221 & \mathbf{17}, 0.282 & 
  \mathbf{22}, 0.362 & \mathbf{31}, 0.497 & \mathbf{48}, 0.645 & \mathbf{60}^+, 
  0.793 & \boldsymbol{\infty}, 0.954 & \boldsymbol{\infty}, 1.097 & 
  \boldsymbol{\infty}, 1.120 \\[1ex]
  \hline
\end{array}$$
\end{table*}

\section{Numerical experiments}
\label{sec:numericalExperiments}

The goal of the numerical experiments is to compare the performance of
the two preconditioners: AGKS and CCMG. The domain is unit square with
a uniform mesh consisting of $2^\ell \times 2^\ell,~ \ell=3,\ldots, 6,$
cells.  The coarsest level mesh contains $8 \times 8$ cells with a
highly-diffusive single island of size $2\times2$ cells centered at
the point $(3/8,3/8)$.  For the discussion of multiple disconnected
islands, refer to~\cite[Sections 3 and 4]{AGKS:2007}. 

We denote the norm of the relative residual at iteration $t$ by $rr^{(t)}$:
\begin{equation*}
rr^{(t)} := \frac{\|r^{(t)}\|_2}{\|r^{(0)}\|_2},
\end{equation*}
where $r^{(t)}$ denotes the residual at iteration $t$ with a stopping
criterion of $rr^{(t)} \leq 10^{-9}.$ In the tables, we report the
iteration count and the average reduction factor of the residual which is
defined as:
\begin{equation*}
\left( rr^{(t)}\right )^{1/t}.
\end{equation*} 
We enforce an iteration bound of $60$. If the the method
seems to converge slightly beyond this bound, we denote
it by $60+$. Whereas, the divergence is denoted by $\infty.$
In Tables \ref{CWI}--\ref{ABS}, iteration count and the average
reduction factor are reported for combinations of 
preconditioner, prolongation, and smoother types. 
\begin{table*}
\caption{Preconditioner = CCMG, Prolongation = Bi-linear, Smoother = ILU}
\label{CBI}      
$$\begin{array}{rrrrrrrrrrr}
  \hline \\
  h \backslash m & 10^0 & 10^1 & 10^2 & 10^3 & 10^4 & 10^5 & 10^6 & 10^7 &
  10^8 & 10^9  \\[1ex]
  \hline\hline \\
  1/8 &~~~ \mathbf{10}, 0.025 &~~~ \mathbf{29}, 0.472 &~~~ \mathbf{30}, 0.500 
  &~~~ \mathbf{35}, 0.528 &~~~ \mathbf{36}, 0.559 &~~~ \mathbf{40}, 0.593 &~~~
  \mathbf{44}, 0.622 &~~~ \mathbf{52}, 0.669 &~~~ \mathbf{52}, 0.669 &~~~
  \mathbf{58}, 0.697 \\[1ex]
  1/16 & \mathbf{3}, 0.000 & \mathbf{3}, 0.000 & \mathbf{3}, 0.001 & \mathbf{6}, 0.003  &  \mathbf{15}, 0.216  & \boldsymbol{\infty}, 0.954  & \boldsymbol{\infty}, 0.954  &  \boldsymbol{\infty}, 0.954  &  \boldsymbol{\infty}, 0.954 &  \boldsymbol{\infty}, 0.954   
    \\[1ex]
  1/32 & \mathbf{3}, 0.000 & \mathbf{3}, 0.000 & \mathbf{4}, 0.003 & 
  \mathbf{7}, 0.045 & \mathbf{19}, 0.316 & \mathbf{60}+, 0.838
  & \mathbf{60}^+, 0.832 &  \mathbf{60}^+, 0.842 &  \mathbf{60}^+, 0.843 & 
  \mathbf{60}^+, 0.843 \\[1ex]
  1/64 & \mathbf{3}, 0.000 & \mathbf{3}, 0.001 & \mathbf{5}, 0.009 &
  \mathbf{7}, 0.028 & \mathbf{18}, 0.290 & \mathbf{16}, 0.239 & \mathbf{10},   0.097 & \mathbf{8}, 0.050 & \mathbf{8}, 0.049 &\mathbf{8}, 0.056 \\[1ex]
  \hline
\end{array}$$
\end{table*}

\begin{table*}
\caption{Preconditioner = CCMG, Prolongation = Bi-linear, Smoother = sGS}
\label{CBS}      
$$\begin{array}[h]{rrrrrrrrrrr}
  \hline \\
  h \backslash m & 10^0 & 10^1 & 10^2 & 10^3 & 10^4 & 10^5 & 10^6 & 10^7 &
  10^8 & 10^9 \\[1ex]
  \hline \hline \\
  1/8 &~~~ \mathbf{10}, 0.025 &~~~ \mathbf{29}, 0.472 &~~~ \mathbf{30}, 0.500 
  &~~~ \mathbf{35}, 0.528 &~~~ \mathbf{36}, 0.559 &~~~ \mathbf{40}, 0.593 &~~~
  \mathbf{44}, 0.622 &~~~ \mathbf{52}, 0.669 &~~~ \mathbf{52}, 0.669 &~~~
  \mathbf{58}, 0.697 \\[1ex]
  1/16 & \mathbf{11}, 0.122 & \mathbf{13}, 0.194 & \mathbf{18}, 0.293 &
  \mathbf{24}, 0.414 & \mathbf{38}, 0.564 & \mathbf{58}, 0.692 & \mathbf{60}^+,
  0.827 & \mathbf{60}^+, 0.931 & \boldsymbol{\infty}, 0.995 &
  \boldsymbol{\infty}, 1.094 \\[1ex]
  1/32 & \mathbf{8}, 0.075 & \mathbf{11}, 0.121 & \mathbf{15}, 0.229 & 
  \mathbf{22}, 0.362 & \mathbf{35}, 0.545 & \mathbf{60}^+, 0.722 &
  \mathbf{60}^+, 0.903 & \boldsymbol{\infty}, 1.022 & \boldsymbol{\infty}, 1.085
  & \boldsymbol{\infty}, 1.133  \\[1ex]
  1/64 & \mathbf{8}, 0.068 & \mathbf{10}, 0.115 & \mathbf{14}, 0.216 &
  \mathbf{19}, 0.334 & \mathbf{27}, 0.449 & \mathbf{44}, 0.600 & \mathbf{60}^+, 
  0.773 & \boldsymbol{\infty}, 0.965 & \boldsymbol{\infty}, 1.116 & 
  \boldsymbol{\infty}, 1.168 \\[1ex]
  \hline
\end{array}$$
\end{table*}
We use Galerkin variational approach to construct the coarser level
algebraic systems. There are two types of prolongation operators under
consideration;  Wesseling-Khalil~\cite{wesselingKhalil1992} and bi-linear,
given by respectively:

\begin{eqnarray*}
  P^{(WK)} & = & \frac{1}{4}
    \left[ \begin{array}{ccccc} 
          1 & 1 & & 0 & 0 \\
          1 & 3 & & 2 & 0 \\
            &   & \cdot & & \\
          0 & 2 & & 3 & 1 \\
          0 & 0 & & 1 & 1
\end{array} \right]_{2h}^{h}, ~~\text{and}~~  R^{(WK)} = P^{(WK)^*}, \\
P^{(B)} & = & \frac{1}{16}
    \left[ \begin{array}{ccccc} 
          1 & 3 & & 3 & 1 \\
          3 & 9 & & 9 & 3 \\
            &   & \cdot & & \\
          3 & 9 & & 9 & 3 \\
          1 & 3 & & 3 & 1
\end{array} \right]_{2h}^{h}, ~~\text{and}~~  R^{(B)} = P^{(B)^*}.
 \end{eqnarray*}
The multigrid preconditioner CCMG is derived from the implementation
by Aksoylu, Bond, and Holst ~\cite{AkBoHo03}.  We employ a
V(1,1)-cycle with point symmetric Gauss-Seidel (sGS) or ILU smoothers.
A direct solver is used for the coarsest level.  We construct two 
different multilevel hierarchies for multigrid preconditioners
$M_{HH}$ in \eqref{practicalAGKS} and $M_{LL}$ in
\eqref{Sherman-Morrison} for DOF corresponding to $\Omega_H$ and $\Omega_L$,
respectively. The corresponding prolongation matrices $P_{HH}$ and
$P_{LL}$ are extracted from the prolongation matrix for whole domain
$\Omega$ in the fashion following \eqref{2x2blockSys}:
\begin{equation*} \label{2x2blockSysProlongation}
P =  
\left[
\begin{array}{ll}
P_{HH} & P_{HL} \\ P_{LH} & P_{LL}
\end{array}
\right].
\end{equation*}
The superior performance of AGKS preconditioner is partially due to
employing these two distinct multilevel hierarchies, which is very much
in spirit of the aforementioned decoupling in Section
\ref{sec:subdomainDefl}.  In fact, due to decoupling, AGKS technology
allows the usage of any ordinary prolongation operator. This operator
does not have to be constructed in a sophisticated manner as in 
the case of Wesseling and Khalil~\cite{wesselingKhalil1992}
or Kwak~\cite{kwak1999}.

As emphasized in our preceding paper~\cite{AGKS:2007}, AGKS can be
used purely as an algebraic preconditioner.  Therefore, the standard
multigrid preconditioner constraint that the coarsest level mesh
resolves the boundary of the island is automatically
eliminated. However, for a fair comparison, we enforce the coarsest
level mesh to have that property.

When the discretization matrix is scaled by $1/m$, we observe a
significant reduction in the iteration count for the AGKS
preconditioner, while, CCMG suffers from inconsistent convergence
behaviour. That is why, we only report the unscaled case for
CCMG. Moreover, for the CCMG preconditioner, we use lexicographic
ordering. On the other hand, for AGKS, we follow the standard way of
ordering the highly-diffusive after the lowly-diffusive DOF as used
in~\cite{AGKS:2007}.

Note that as the diffusion coefficient $m$ increases, the CCMG method
becomes less effective.  For sufficiently large $m$, it even diverges.
CCMG shows this adverse behaviour for all types of transfer operators,
and smoother types for almost all levels.  From Tables
\ref{CWI}--\ref{CBS}, CCMG preconditioner under ILU smoother provides
about three times faster convergence than that under sGS smoother.
However, we also observe that the CCMG preconditioner becomes totally
ineffective even diverges for $m \geq 10^5$ under both types of
smoothers; see the corresponding columns in Tables
\ref{CWI}--\ref{CBS}.

On the other hand, the AGKS preconditioner becomes more effective with
increasing $m$. Even for the smallest $m$ value, AGKS performance
still remains comparable to the CCMG performance.  Furthermore, as
seen in Tables \ref{AWI}--\ref{ABS}, AGKS preconditioner demonstrates
consistently similar iteration counts and contraction numbers for all
types of transfer operators and smoothers.  Therefore, AGKS
performance is independent of the utilized prolongation operators and
smoothers.

CCMG performance heavily depends on the smoother choice. As opposed to
expectations from the vertex-centered case, we observe that $x$- and
$y$-line smoothers do not improve CCMG performance compared to sGS and
ILU smoothers; see Llorente and Melson~\cite{Llorente.I;Melson.N2000}
for other ordering related smoother complications.  As pointed out by
Mohr and Wienands~\cite{mohrWienands2004}, CCMG may need a sophisticated
smoother like ILU. The specific ILU choice such as no-fill-in and
sparsity threshold dramatically changes CCMG iteration counts. We
use a sparsity threshold of $10^{-5}$ and observe that it improves the
iteration count by a factor of $3$ compared to the no-fill-in
case. The behaviour of ILU was extensively studied by
Wittum~\cite{wittum1989a,wittum1989c}, also 
see~\cite[pp. 98 and 134]{wesseling1988}.

\begin{table*}
\caption{
Preconditioner = AGKS, Prolongation = Wesseling-Khalil, Smoother = ILU}
\label{AWI}
$$
\begin{array}{rrrrrrrrrrrrrrr}
  \hline \\
  h \backslash m & 10^0 & 10^1 & 10^2 & 10^3 & 10^4 & 10^5 & 10^6 & 10^7 & 
  10^8 & 10^9 & 10^{11} & 10^{13}  \\[1ex]
  \hline \hline \\
  1/8  &~~ \mathbf{22}, 0.371 &~~ \mathbf{10}, 0.115 &~~ \mathbf{10}, 0.116 &~~
  \mathbf{9}, 0.078 &~~ \mathbf{9}, 0.056 &~~ \mathbf{8}, 0.059 &~~ \mathbf{8}, 
  0.045 &~~ \mathbf{8}, 0.039 &~~ \mathbf{8}, 0.039 &~~ \mathbf{8}, 0.039 &~~
  \mathbf{8}, 0.039 &~~ \mathbf{8}, 0.039   \\[1ex]
  1/16 &\mathbf{16}, 0.240 &\mathbf{13}, 0.199 & \mathbf{11}, 0.131 & \mathbf{9},
  0.098 & \mathbf{8}, 0.043 & \mathbf{7}, 0.043 & \mathbf{6}, 0.018
  & \mathbf{6}, 0.011 & \mathbf{6}, 0.007 & \mathbf{6}, 0.010 & \mathbf{5},
  0.010  & \mathbf{5}, 0.006\\[1ex]
  1/32 & \mathbf{20}, 0.329 & \mathbf{19}, 0.313 & \mathbf{14}, 0.192 & 
  \mathbf{11}, 0.127 & \mathbf{9}, 0.093 & \mathbf{8}, 0.035 & \mathbf{7}, 
  0.041 & \mathbf{6}, 0.018 & \mathbf{6}, 0.012 & \mathbf{6}, 0.008 &
  \mathbf{6},  0.004 & \mathbf{5}, 0.008  \\[1ex]
  1/64 & \mathbf{29}, 0.484 & \mathbf{26}, 0.435 & \mathbf{17}, 0.284 & 
  \mathbf{12}, 0.161 & \mathbf{10}, 0.092 & \mathbf{8}, 0.061 & \mathbf{8}, 
  0.026 & \mathbf{6}, 0.031 & \mathbf{6}, 0.016 & \mathbf{6}, 0.010 &
  \mathbf{6}, 0.006 & \mathbf{5}, 0.013 \\[1ex]
  \hline
\end{array} $$
\end{table*}

\begin{table*}
\caption{
Preconditioner = AGKS, Prolongation = Wesseling-Khalil, Smoother = sGS}
\label{AWS}      
$$\begin{array}{rrrrrrrrrrrrrrr}
  \hline \\
  h \backslash m & 10^0 & 10^1 & 10^2 & 10^3 & 10^4 & 10^5 & 10^6 & 10^7 &
  10^8 & 10^9 & 10^{11} & 10^{13}  \\[1ex]
  \hline \hline \\
  1/8  &~~ \mathbf{22}, 0.371 &~~ \mathbf{10}, 0.115 &~~ \mathbf{10}, 0.116 &~~
  \mathbf{9}, 0.078 &~~ \mathbf{9}, 0.056 &~~ \mathbf{8}, 0.059 &~~ \mathbf{8}, 
  0.045 &~~ \mathbf{8}, 0.039 &~~ \mathbf{8}, 0.039 &~~ \mathbf{8}, 0.039 &~~
  \mathbf{8}, 0.039 &~~ \mathbf{8}, 0.039   \\[1ex]
  1/16 &\mathbf{16}, 0.268 & \mathbf{13}, 0.201 & \mathbf{11}, 0.131 &\mathbf{9}, 0.098 & 
  \mathbf{8}, 0.043 & \mathbf{7}, 0.043 & \mathbf{6}, 0.017 & \mathbf{6}, 0.010
  & \mathbf{6}, 0.007 & \mathbf{6}, 0.004 & \mathbf{5}, 0.010 & \mathbf{5},
  0.005\\[1ex]
  1/32 & \mathbf{20}, 0.350 & \mathbf{19}, 0.317 & \mathbf{14}, 0.192 & 
  \mathbf{11}, 0.127 & \mathbf{9}, 0.093 & \mathbf{8}, 0.035 & \mathbf{7},
  0.041   & \mathbf{6}, 0.016 & \mathbf{6}, 0.010 & \mathbf{6}, 0.007 &
  \mathbf{6},
  0.008 & \mathbf{5}, 0.008\\[1ex]
  1/64 & \mathbf{29}, 0.483 & \mathbf{26}, 0.436 & \mathbf{17}, 0.283 &
 \mathbf{12}, 0.162 & \mathbf{10}, 0.092 & \mathbf{8}, 0.061 & \mathbf{8}, 
0.030 & \mathbf{6}, 0.031 & \mathbf{6}, 0.017 & \mathbf{6}, 0.011 & \mathbf{6}, 0.005 & \mathbf{5}, 0.013 \\[1ex]
  \hline
\end{array}$$
\end{table*}

We observe an interesting cut-off $m$ value for performances of
preconditioners.  While, CCMG performance starts to deteriorate at
about $m = 10^5$, the AGKS preconditioner reaches its peak performance
and maintains an optimal iteration count for all $m \geq 10^5$. For
instance, the iteration counts for the CCMG method jumps to $60$ at
about $m = 10^5$ except for level four with ILU smoothing
case. For that level, the iteration count starts to decrease although
the numerical solution does not converge to the exact
solution. Therefore, the performance of the CCMG preconditioner gets
worse after $m = 10^5$. Consequently, with respect to the magnitude of
the coefficient contrast, this observation indicates that CCMG fails
to be robust whereas AGKS maintains its robustness.

\begin{table*}
\caption{Preconditioner = AGKS, Prolongation = Bi-linear, Smoother = ILU}
\label{ABI}      
$$\begin{array}{rrrrrrrrrrrrrrr}
  \hline \\
  h \backslash m & 10^0 & 10^1 & 10^2 & 10^3 & 10^4 & 10^5 & 10^6 & 10^7 &
  10^8 & 10^9 & 10^{11} & 10^{13} \\[1ex]
  \hline \hline \\
  1/8  &~~ \mathbf{22}, 0.371 &~~ \mathbf{10}, 0.115 &~~ \mathbf{10}, 0.116 &~~
  \mathbf{9}, 0.078 &~~ \mathbf{9}, 0.056 &~~ \mathbf{8}, 0.059 &~~ \mathbf{8}, 
  0.045 &~~ \mathbf{8}, 0.039 &~~ \mathbf{8}, 0.039 &~~ \mathbf{8}, 0.039 &~~
  \mathbf{8}, 0.039 &~~ \mathbf{8}, 0.039   \\[1ex]
  1/16 & \mathbf{16}, 0.237 & \mathbf{13}, 0.197 & \mathbf{11}, 0.131 & \mathbf{9}, 0.098 &
  \mathbf{8}, 0.043 & \mathbf{7}, 0.043 & \mathbf{6}, 0.018 & \mathbf{6}, 0.012
  & \mathbf{6}, 0.007 & \mathbf{6}, 0.011 & \mathbf{5}, 0.010 & \mathbf{5},
  0.006\\[1ex]
  1/32 & \mathbf{20}, 0.348 & \mathbf{19}, 0.312 & \mathbf{14}, 0.192 & 
  \mathbf{11}, 0.126 & \mathbf{9}, 0.093 & \mathbf{8}, 0.035 & \mathbf{7}, 0.041
  & \mathbf{6}, 0.016 & \mathbf{6}, 0.010 & \mathbf{6}, 0.006 & \mathbf{6}, 
  0.021 & \mathbf{5}, 0.008 \\[1ex]
  1/64 & \mathbf{29}, 0.483 & \mathbf{26}, 0.431 & \mathbf{17}, 0.283 & 
  \mathbf{12}, 0.161 & \mathbf{10}, 0.092 & \mathbf{8}, 0.062 & \mathbf{8},
  0.026  & \mathbf{6}, 0.030 & \mathbf{6}, 0.015 & \mathbf{6}, 0.010 & 
  \mathbf{6}, 0.006 & \mathbf{5}, 0.013  \\[1ex]
  \hline
\end{array}$$
\end{table*}

\begin{table*}
\caption{Preconditioner = AGKS, Prolongation = Bi-linear, Smoother = sGS}
\label{ABS}      
$$\begin{array}{rrrrrrrrrrrrrrr}
  \hline \\
  h \backslash m & 10^0 & 10^1 & 10^2 & 10^3 & 10^4 & 10^5 & 10^6 & 10^7 &
  10^8 & 10^9 & 10^{11} & 10^{13} \\[1ex] 
  \hline \hline \\
  1/8  &~~ \mathbf{22}, 0.371 &~~ \mathbf{10}, 0.115 &~~ \mathbf{10}, 0.116 &~~
  \mathbf{9}, 0.078 &~~ \mathbf{9}, 0.056 &~~ \mathbf{8}, 0.059 &~~ \mathbf{8}, 
  0.045 &~~ \mathbf{8}, 0.039 &~~ \mathbf{8}, 0.039 &~~ \mathbf{8}, 0.039 &~~
  \mathbf{8}, 0.039 &~~ \mathbf{8}, 0.039   \\[1ex]
  1/16 &\mathbf{16}, 0.256 & \mathbf{13}, 0.197 & \mathbf{11}, 0.131 & \mathbf{9}, 0.098 & 
  \mathbf{8}, 0.043 & \mathbf{7}, 0.043 & \mathbf{6}, 0.017 & \mathbf{6}, 0.010
  & \mathbf{6}, 0.007 & \mathbf{6}, 0.004 & \mathbf{5}, 0.010 & \mathbf{5}, 
  0.004 \\[1ex]
  1/32 & \mathbf{20}, 0.352 & \mathbf{19}, 0.317 & \mathbf{14}, 0.192 & 
  \mathbf{11}, 0.127 & \mathbf{9}, 0.093 & \mathbf{8}, 0.035 & \mathbf{7}, 0.041
  & \mathbf{6}, 0.016 & \mathbf{6}, 0.010 & \mathbf{6}, 0.007 & \mathbf{6},
  0.020 & \mathbf{5}, 0.009 \\[1ex]
  1/64 & \mathbf{29}, 0.481 & \mathbf{26}, 0.438 & \mathbf{17}, 0.281
  & \mathbf{12}, 0.161 & \mathbf{10}, 0.092 & \mathbf{8}, 0.061 & \mathbf{8},
  0.026 & \mathbf{6}, 0.031 & \mathbf{6}, 0.018 & \mathbf{6}, 0.012 &
  \mathbf{6}, 0.008 & \mathbf{5}, 0.014 \\[1ex]
  \hline
\end{array}$$
\end{table*}

\end{document}